\documentclass[a4paper,11pt]{article}

\usepackage[margin=1.1in]{geometry}
\usepackage{amsfonts,amsmath,amsthm}
\usepackage[latin1]{inputenc}
\usepackage{textcomp}
  \usepackage{version,amsfonts,amsmath,boxedminipage,alltt}
  \usepackage{graphics}
 \usepackage{graphicx}

\newtheorem{cor}{Corollary}

\newtheorem{theorem}{Theorem}
\newtheorem{lemma}{Lemma}
\theoremstyle{remark}

\let\epsilon=\varepsilon

\let\phi=\varphi

 \let\Th=\Theta

\let\tilde=\widetilde

\newcommand{\E}{\field{E}}

\newcommand{\field}[1]{\mathbb{#1}}
\newcommand{\R}{\field{R}}

\newcommand{\e}{{\mathcal E}}

\newcommand{\D}{\mathcal{D}}

\newcommand{\fracn}{\mbox{$\frac{1}{n}$}}

\def\der^#1_#2{\frac{\partial^{#1}}{\partial {#2}^{#1}}}

\RequirePackage[OT1]{fontenc}

\RequirePackage{amsthm,amsmath}

\def\1{\mbox{1\hspace{-.25em}I}}

\def\t{\theta}

\def\D{\Delta}
\def\Th{\Theta}
\def\e{\varepsilon}

\def\bR{\mathbb{R}}
\def\bP{\mathbb{P}}

\title{Linear and Conic Programming Estimators in High-Dimensional Errors-in-variables Models}
\author{Alexandre Belloni\footnote{
The Fuqua School of Business,
    Duke University}, \, Mathieu Rosenbaum\footnote{Laboratoire de Probabilit\'es et Mod\`eles Al\'eatoires,  Universit\'e Pierre et Marie Curie (Paris 6)} \footnote{Centre de Recherche en Economie et Statistique,  ENSAE-Paris Tech} \, and Alexandre B. Tsybakov$^\ddag$
    }
\begin{document}
\maketitle
\begin{abstract}
\noindent

We consider the linear regression model with observation error in the design. In this setting,
we allow the number of covariates to be much larger than the sample size. Several new estimation
methods have been recently introduced for this model. Indeed, the standard Lasso estimator or Dantzig selector turn out to become unreliable when only noisy regressors are available, which is quite common in practice. In this work, we propose and analyse a new estimator for the errors-in-variables model. Under suitable sparsity assumptions, we show that this estimator attains the minimax efficiency bound. Importantly, this estimator can be written as a second order cone programming minimisation problem which can be solved numerically in polynomial time. Finally, we show that the procedure introduced in \cite{RT2}, which is almost optimal in a minimax sense, can be efficiently computed by a single linear programming problem despite non-convexities.
\end{abstract}

\section{Introduction}
\noindent We consider the regression model with observation error in the
design:
\begin{eqnarray}
y= X\theta^* + \xi, \quad \quad Z=X+W\label{model1}.
\end{eqnarray}
Here the random vector $y\in\mathbb{R}^n$ and the random $n\times p$ matrix
$Z$ are observed, the $n\times p$ matrix $X$ is unknown, $W$ is an
$n\times p$ random noise matrix, $\xi\in \mathbb{R}^n$ is a random noise
vector, and $\theta^*\in \mathbb{R}^p$ is a vector of unknown parameters to be
estimated. For example, the case where the entries of matrix $X$ are missing
at random can be reduced to this model. Such linear regressions with errors in both variables have been
widely investigated in the literature, see for example \cite{cb90,cvn99,ks73}. Our work is different in that we consider the setting where the dimension $p$ can be much larger than the sample size $n$, and $\theta^*$ is sparse.

It has been shown in \cite{RT1} that the presence of observation noise has severe consequences on
the usual estimation procedures in the high-dimensional setting. In particular, the Lasso estimator and Dantzig
selector turn out to be inaccurate and fail to identify the sparsity pattern of the vector $\theta^*$. The same paper provides an alternative procedure, called Matrix Uncertainty selector (MU selector for short),
which is robust to the presence of noise. The MU~selector $\hat \theta^{MU}$ is defined as a
solution of the minimisation problem
\begin{equation}\label{Def:MU}
\min\{ |\theta|_1: \,\, \theta\in\Theta, \,
\big|\fracn Z^T(y-Z\theta)\big|_\infty\leq \mu |\theta|_1 +\tau\},
\end{equation}
where $|\cdot|_{q}$ denotes the $\ell_q$-norm for $1\le q\leq\infty$,
$\Theta$ is a given convex subset of $\mathbb{R}^p$ characterising the prior
knowledge about $\theta^*$, and the constants $\mu\geq 0$ and $\tau\geq 0$
depend on the level of the noises $W$ and $\xi$ respectively. An extension of the MU selector to generalized linear models is discussed in \cite{SFT14}.

In \cite{RT2}, a modification of the MU selector is suggested. It applies when $W$ is a random matrix with independent and zero mean
entries $W_{ij}$ such that for any $1\leq j\leq p$, the sum of expectations
$$\sigma_j^2=\frac{1}{n}\sum_{i=1}^n \mathbb{E}(W_{ij}^2)$$
is finite and admits a data-driven estimator. This is for example the case in the model
with missing data: $\tilde Z_{ij} = X_{ij}\eta_{ij},~i=1,\dots,n,\ j=1,\dots,p,$
where for each fixed $j=1,\dots,p$, the factors $\eta_{ij},
i=1,\dots,n,$ are i.i.d. Bernoulli random variables taking the value 1
with probability $1-\pi_j$ and 0 with probability $\pi_j$,
$0<\pi_j<1$. Indeed, this model can be rewritten under the form
$Z_{ij} = X_{ij}+ W_{ij},$
where $Z_{ij}={\tilde Z}_{ij}/(1-\pi_j)$ and
$W_{ij}=X_{ij}(\eta_{ij}-(1-\pi_j))/(1-\pi_j)$. Thus, in this model, the $\sigma_j^2$ satisfy
$\sigma_j^2 = \fracn\sum_{i=1}^n X_{ij}^2 \,
\frac{\pi_j}{1-\pi_j},$ and it is easily shown that they admit good data-driven estimators $\hat\sigma_j^2$, see \cite{RT2}. There are of course other examples where $\sigma_j^2$ can be accurately estimated from the data.  One of them is related to repeated measurement models where the values of $Z$ are available on a finer time scale. We refer here to examples given in \cite{RT1}, such as portfolio replication. In the problem of portfolio replication, assuming for example that the entries of $X$ are approximately constant on the daily scale, we can readily estimate the variances using the additional finer scale measurements of $Z$ on the hourly scale.

The construction of this modified estimator is based on the following idea.
We cannot use $X$ in our estimation procedure since only its noisy version $Z$ is available.
In particular, the MU selector involves the matrix $Z^TZ/n$ instead of $X^TX/n$.
Compared to $X^TX/n$, this matrix contains a bias induced by the
diagonal entries of the matrix $W^TW/n$ whose expectations
$\sigma_j^2$ do not vanish. Therefore, if the $\sigma_j^2$ can be estimated, a natural idea is to compensate
this bias thanks to these estimates. This leads to a
new estimator $\hat\theta^{C}$, called compensated MU selector, and defined as a solution of the
minimisation problem
\begin{equation} \label{C}
\min\{|\theta|_1: \,\, \theta\in\Theta, \,
\big|\fracn Z^T(y-Z\theta)+\widehat{D}\theta\big|_\infty\leq\mu|\theta|_1
+\tau\},
\end{equation}
where $\widehat{D}$ is the diagonal matrix with entries
$\hat{\sigma}_j^2$, which are estimators of $\sigma_j^2$, and
$\mu\geq 0$ and $\tau\geq 0$ are constants chosen according to the level of the noises and the accuracy of the estimators $\hat\sigma_j^2$.

Several aspects of the compensated MU selector are studied in \cite{RT2}, in particular the rates of convergence in $\ell_q$-norm,
the prediction risk and the design of confidence intervals. One of the interests of this modification of the MU selector is that it enables us to obtain bounds for the estimation errors which are decreasing with the sample size $n$. This is in contrast to the case of the MU selector, where the bounds are small
only if the noise $W$ is small. 
For example, if $\t^*$ is $s$-sparse, it is shown in \cite{RT2} that under appropriate assumptions,
\begin{equation}\label{rt13}
|\hat\theta^{C}-\t^*|_q\le C s^{1/q}\sqrt{\frac{\log p}{n}}(|\t^*|_1+1), \quad  1\le q\le \infty,
\end{equation}
with probability close to 1, where $C>0$ is a constant independent of $s,p,n$, and $\theta^*$.

An alternative Lasso type estimator (non-convex program) complemented by an iterative relaxation procedure is introduced in \cite{LW}.
This method requires the knowledge of the exact value of $|\t^*|_1$ (or of the property $|\t^*|_1\le b\sqrt{s}$ for a constant $b$), and of an upper bound on $|\t^*|_2$. Considering the setting where the entries of the regression matrix $X$ are zero-mean subgaussian, it is shown in \cite{LW} that if $\t^*$ is $s$-sparse, under appropriate assumptions, the resulting estimator $\hat \t'$ satisfies
\begin{equation}\label{lw12}
|\hat \t'-\t^*|_2\le C(\t^*) s^{1/{2}}\sqrt{\frac{\log p}{n}}(|\t^*|_{2}+1),
\end{equation}
with probability close to 1, where $C(\t^*)>0$ depends on $\t^*$ in a non-specified way. Related covariates selection results are reported in  \cite{SFT12}. In \cite{CChen,CChen1}, the authors propose yet another method of estimation of $\t^*$, based on orthogonal matching pursuit (OMP). Their procedure
needs the parameter $s$ (the exact number of non-zero components of $\t^*$) as an input. Moreover, they impose the additional assumption that the non-zero components $\t^*_j$ of $\t^*$ are sufficiently large:
\begin{equation}\label{cc12}
|\t^*_j| \ge c \sqrt{\frac{\log p}{n}}(|\t^*|_2+1), \quad j=1,\dots, p,
\end{equation}
where $c>0$ is a constant. Focusing as in \cite{LW} on the case where the entries of the regression matrix $X$ are zero-mean subgaussian, it is shown in \cite{CChen,CChen1} that the OMP estimator satisfies a bound analogous to (\ref{lw12}) with constant $C(\t^*)\equiv C>0$ independent of $\t^*$, as well as a consistent support recovery property.

These recent developments shed light on errors-in-variables problems in high dimensional settings. However, they are not fully satisfying. Indeed, the following issues are remaining:

\vspace{1mm}

\noindent$\bullet$ From a practical viewpoint, the use of the above estimators can be intricate. In particular, the minimisation problem (\ref{C}) is not always a convex one, and \cite{RT2} does not provide an algorithm enabling to solve it in general case. Although the methods suggested in \cite{LW} and \cite{CChen,CChen1} are computationally feasible, they need the knowledge of the parameters $|\t^*|_1$, $|\t^*|_2$ or $s$, which are not available in practice.

\vspace{1mm}

\noindent$\bullet$ While the bound (\ref{rt13}) is more general than (\ref{lw12}) (it holds for all $q$ and not only for zero-mean subgaussian $X$), it is less accurate than (\ref{lw12}) in the case $q=2$ assuming that (\ref{lw12}) is established with $C(\t^*)\equiv C>0$ independent of $\t^*$. Indeed, $|\t^*|_2$ is always smaller than $|\t^*|_1$. For example, if all components of $\t^*$ take the same value and  $\t^*$ is $s$-sparse, then $|\t^*|_2 = |\t^*|_1/\sqrt{s}$. In fact, the optimal rate of convergence in $\ell_q$-norm on the class of $s$-sparse vectors, as a function of $s,p,n$ and the norms $|\t^*|_r$, remains unknown. When $q=2$ and $X$ is zero-mean Gaussian, a minimax lower bound including the factor $|\t^*|_2$ and not $|\t^*|_1$ is stated without proof in \cite{CChen1}. This, however,
does not answer the question in general situation.

\vspace{1mm}

 The aim of this paper is to provide answers to the above two questions. It is organized as follows. After giving some definitions and assumptions in Sections \ref{sec:assump} and \ref{sec:def}, we introduce in Section \ref{sec:conic} a new estimator $\hat \t$ which is based on second order cone programming and thus can be computed in polynomial time. We show that, under appropriate conditions, this estimator attains bounds of the form
\begin{equation}\label{rt14}
|\hat \t-\t^*|_q\le C s^{1/{q}}\sqrt{\frac{\log p}{n}}(|\t^*|_{2}+1), \quad 1\le q\le \infty,
\end{equation}
with probability close to 1, where the constant $C$ does not depend on $s,p,n$ and $\t^*$.  Contrary to the procedures of \cite{CChen,CChen1} and \cite{LW}, this new estimator does not require the knowledge of $|\t^*|_1$, $|\t^*|_2$ or $s$ to be computed. We also do not need a lower bound condition such as (\ref{cc12}) on the components of the target vector $\t^*$. Another difference from the mentioned papers is that our main results do not focus on zero-mean subgaussian regression matrices $X$, but rather deal with deterministic matrices $X$ commonly appearing in applications. As an easy consequence, we show that the results extend to random matrices $X$ 
by using suitable deviation properties for the quantity
\begin{equation*}
m_2=\max_{j=1,\dots,p} \frac1{n}\sum_{i=1}^n X_{ij}^2,
\end{equation*}
where the $X_{ij}$ are the entries of $X$, as well as checking a Restricted Eigenvalue type condition on the matrix $X^TX/n$. This extension is possible under the assumption that $X$ is independent of $\xi$ and $W$. Note that under mild moment conditions, $m_2$ can be suitably controlled with high probability, see Section 5 of the Supplementary Material. Sufficient conditions for Restricted Eigenvalue type conditions are well known in the high-dimensional setting, see \cite{BRT09}.

Furthermore, in Section \ref{sec:lower}, we prove minimax lower bounds  showing that no estimator can achieve faster rate than that given in (\ref{rt14}), up to a logarithmic in $s$ factor, uniformly on a class of $s$-sparse vectors. Finally, Section \ref{sec:simulations} provides simulation results that compare the conic estimator, the compensated MU selector and variants of the Dantzig selector. The results are in accordance with the theoretical findings.

While the conic programming estimator solves a tractable convex minimisation problem, the compensated MU selector is in general a non-convex program. Section \ref{sec:ComputationC} in the Supplementary Material is devoted to address this issue. We show that under mild assumptions, the compensated MU selector can be reduced to convex programming. In fact, when $\Theta=\R^p$ or $\Theta$ is defined by linear constraints, it can even be written as a single linear programming problem, which is of course a computational advantage compared to the estimator based on conic programming. However, the rate of convergence of the compensated MU selector is suboptimal. The proofs are relegated to the appendices and some additional results and technical lemmas are given in the Supplementary Material.


\section{Assumptions on the model}\label{sec:assump}
In this section, we introduce the assumptions that will be used below 
to study the statistical properties of the estimators. Recall that for $\gamma>0$, the random variable $\eta$ is said to be
{\it subgaussian with variance parameter} $\gamma^2$ (or shortly {\it $\gamma$-subgaussian}) if, for all $t\in\mathbb{R}$,
\begin{equation*}
\E[\text{exp}(t\eta)]\leq
\text{exp}(\gamma^2t^2/2).
\end{equation*}
A random vector $\zeta\in \R^p$ is said to be
{\it subgaussian with variance parameter} $\gamma^2$ (or shortly {\it $\gamma$-subgaussian}) if the inner products $(\zeta, v)$ are  $\gamma$-subgaussian for any $v\in \R^p$ with $|v|_2=1$. We shall consider the following assumptions.

\smallskip
{\it
\begin{itemize}
\item[(A1)] The matrix $X$ is deterministic.
\item[(A2)] The elements of the random vector $\xi$ are independent zero-mean subgaussian random variables
with variance parameter~$\sigma^2$.
\item[(A3)] The rows $w_i$, $i=1,\dots,n$, of the noise matrix $W$ are independent zero-mean subgaussian random vectors with variance parameter~$\sigma^2_*$, and $\E(W_{ij}W_{ik})=0$ for all $1\le j <k \le p$. Furthermore, $W$ is independent of $\xi$.

\item[(A4)] There exist statistics $\hat{\sigma}_j^2$ such that for any $\varepsilon> 0$, we have
\begin{equation}\label{pr1}
\mathbb{P}\big[\max_{j=1,\dots,p}|\hat{\sigma}_j^2-\sigma_j^2|\geq
b(\e)\big]\leq \varepsilon,
\end{equation}
where  $b(\varepsilon)= c_b\sqrt{\frac{\log(c_b'p/\varepsilon)}{n}}$ for some constants $c_b>0$ and $c_b'>0$.
\end{itemize}
}
\noindent Assumptions (A1) -- (A3) are quite standard. Note that we do not assume independence of the components of each $w_i$. Examples of sufficient conditions for (A4) in the model with missing data are provided in~\cite{RT2}.\\

\section{Sensitivities}\label{sec:def}

It is well known, see for example \cite{BRT09}, that the performance of Lasso or Dantzig selector type estimators in high-dimensional linear models is determined by specific characteristics of the Gram
matrix
$$
\Psi = \frac1{n}X^TX,
$$
such as the Restricted Eigenvalue constants. We shall need similar characteristics here.  Following  \cite{GT}, we define them in a more general form, so that the required property is a consequence of the Restricted Eigenvalue property whenever the latter is satisfied.
For a vector $\t$ in $\bR^p$, we denote by $\t_J$ the vector in $\bR^p$ that
has the same coordinates as $\t$ on the set of indices $J\subset
\{1,\hdots,p\}$ and zero coordinates on its complement~$J^c$. We
denote by $|J|$ the cardinality of $J$.

For any $u>0$ and any subset $J$ of $\{1,\hdots,p\}$, consider 
the cone
$$
C_{J}(u)=\big\{\Delta\in\bR^p:\ |\Delta_{J^c}|_1\le
u|\Delta_{J}|_1 \big\}.
$$
The use of such cones to define the Restricted Eigenvalue constants and other related characteristics of the Gram matrix is standard in the literature on the Lasso and Dantzig selector, starting from \cite{BRT09}. 
For $q\in[1,\infty]$ and
an integer $s\in[1,p]$, the paper \cite{GT} defines the {\em $\ell_q$-sensitivity} as
follows:
$$
\kappa_{q}(s,u)=\min_{J: \ |J|\le s} \Big(\min_{\Delta\in
C_J(u):\ |\Delta|_q=1} \left|\Psi\Delta \right|_{\infty}\Big).
$$
In \cite{GT,GT2}, it is shown that meaningful bounds  for various types of errors in sparse linear regression can be obtained
in terms of the sensitivities $\kappa_{q}(s,u)$. In particular,
it is proved in \cite{GT} that the approach based on sensitivities is more general
than that based on the Restricted Eigenvalue or the
Coherence condition. In particular, under those assumptions, $$\kappa_{q}(s,u) \ge c s^{-1/q},$$
for some constant $c>0$, which implies the rate optimal bounds for the errors of Lasso and Dantzig selector estimators as in \cite{BRT09}. For convenience, some properties of $\kappa_{q}(s,u)$ proved in~\cite{GT} are summarized in Section \ref{appendixC} of the Supplementary Material.

In addition to $\kappa_{q}(s,u)$, we introduce a {\it prediction sensitivity} as follows:
$$ \kappa_{\rm pr}(s,u) = \min_{J:|J|\leq s} \Big( \min_{\Delta \in C_J(u):|\Psi^{1/2}\Delta|_2=1}|\Psi\Delta|_\infty\Big).$$
The sensitivity $\kappa_{\rm pr}(s,u)$ is useful to establish convergence in the prediction norm with fast rates, see \eqref{Relax11c} in Theorem~\ref{th:conicRelaxed} below. Such rates can be obtained under more general assumptions than rates of convergence in $\ell_q$-norm. A discussion of the case of repeated regressors is given in~\cite{BCW2014}.  The prediction sensitivity is closely related to the identifiability factor defined in \cite{CCK2013}. Lemma~\ref{lem:appC2} in Section \ref{appendixC} of the Supplementary Material shows that $\kappa_{\rm pr}(s,u)>0$ quite generally. Also, $\kappa_{\rm pr}(s,u) \geq \sqrt{\kappa_1(s,u)}$, see Lemma~\ref{lem:appC1} in Section \ref{appendixC} in the Supplementary Material.

\section{Estimator based on conic programming}
\label{sec:conic}
In this section, we introduce our conic programming based estimator $\hat\theta$. This estimator is computationally feasible and we provide upper bounds on its estimation and prediction errors. It will be shown in Section \ref{sec:lower} that these bounds cannot be improved in a minimax sense. In what follows, we fix a (small) value $\e>0$.  The probabilities, with which the bounds on the estimation and prediction errors hold will be of the form $1-c\e$ for some $c>0$.

To define the estimator $\hat\theta$, we consider the following minimisation problem:
\begin{eqnarray}\label{conic}
&&{\rm minimise}  \ |\theta|_1+ \lambda t
\ \ \ \  {\rm over} \ (\theta,t) \ {\rm such \ that:}
\end{eqnarray}
$$
\theta\in\Theta, \, t \in \R^+,
\big|\fracn Z^T(y-Z\theta)+\widehat{D}\theta\big|_\infty\leq\mu t
+\tau, \ |\theta|_2 \le t.
$$
Here $\lambda$, $\mu$ and $\tau$ are positive tuning constants, and $\Theta$ is a given subset of $\R^p$ characterising the prior knowledge about $\theta$. In the results below,
$\lambda$, $\mu$ and $\tau$  are of the form
$$\lambda \in [1/2, 2], \quad \mu=C\sqrt{\frac{\log(p/\e)}{n}}, \quad \tau=C\sqrt{\frac{\log(p/\e)}{n}},$$
where we denote by $C>0$ constants depending only on $m_2$ and on the constants appearing in Assumptions (A1) -- (A4). More specifically, in the theory we take,
\begin{equation}\label{DefLambdaMu}
\mu=\delta_1'(\e)+\delta_4'(\e)+\delta_5(\e) + b(\e),~~\tau=\delta_2(\e)+\delta_3(\e)
\end{equation}
where $\delta_i(\e)$ and $\delta_i'(\e)$ are defined in Lemmas \ref{lem2} and \ref{lem3a} in Appendix A.

When $\Theta=\R^p$ or $\Theta$ is a subset of $\R^p$ defined by linear constraints, \eqref{conic} is a conic programming problem. Therefore it can be efficiently solved in polynomial time. 

Let $(\hat\theta, \hat t)$ be a solution of (\ref{conic}). We take $\hat\theta$ as estimator of $\theta^*$. Under Assumptions (A1) -- (A4), it follows that the feasible set of the minimisation problem (\ref{conic}) is not empty with high probability if $\varepsilon$ is small enough (see Lemma \ref{lem3} in Appendix B).

The following theorem, proved in Appendix B, is our main result about the statistical properties of the estimator  $\hat\t$ based on conic programming.

\begin{theorem}\label{th:conic}
Assume (A1)--(A4), and that the true parameter $\t^*$ is $s$-sparse and belongs to $\Th$. Let $\e>0$, $1\le q\le \infty$, and set $\mu$ and $\tau$ as in (\ref{DefLambdaMu}).
Assume also that
\begin{equation}\label{ass_kappa}\kappa_q(s,1+\lambda)\ge cs^{-1/q},
\end{equation}
 for some constant $c>0$ and that
 \begin{equation}\label{hyps}
 s\le c_1(\lambda^{-1}+\lambda)^{-1}\sqrt{n/\log (p/\e)},\end{equation} for some small enough constant $c_1>0$. Then, with probability at least $1-8\e$,
\begin{eqnarray}\label{11a}
|\hat\t-\t^*|_q&\le& C s^{1/q} \sqrt{\frac{\log (c'p/\e)}{n}} (|\theta^*|_2 +1),
\end{eqnarray}
for some constants $C>0$ and $c'>0$ (here and in the sequel we set $s^{1/\infty}=1$).\\

\noindent Under the same assumptions, the prediction risk admits the following bound,
with probability at least $1-8\e$:
\begin{eqnarray}\label{11c}
\fracn\big|X(\hat\t-\t^*)\big|_2^2&\le& C s \frac{\log (c'p/\e)}{n} (|\theta^*|_2 +1)^2\,.
\end{eqnarray}
The constants $C>0$ and $c'>0$ in \eqref{11a} and \eqref{11c} depend only on $m_2$ and on the constants appearing in Assumptions (A1)--(A4).
\end{theorem}


Some remarks are in order here. The results in Theorem \ref{th:conic} highlight the impact of $\lambda$ and suggest that we should set $\lambda \asymp 1$. Theorem \ref{th:conic} is established under the condition $\kappa_q(s,1+\lambda)\ge cs^{-1/q}$, which holds under standard assumptions on the matrix $X$.  For example, it holds simultaneously for all $q$ under the Coherence assumption, see \eqref{k2} in the Supplementary Material. For $1\le q \le 2$ this condition follows from the Restricted Eigenvalue (RE) assumption, see \eqref{k3}, \eqref{k4} in the Supplementary Material. It is shown in \cite{RZ13} that the RE assumption is satisfied with high probability for a large class of random matrices, including random matrices with zero-mean subgaussian rows and non-trivial covariance structure, as well as matrices with
zero-mean independent rows and uniformly bounded entries. This allows us to extend Theorem \ref{th:conic} to random matrices $X$ as follows. Assume that the conditional distribution of $(\xi, W)$ given $X$ is such that (A2)--(A4) are satisfied conditionally on $X$ for some fixed constants $c_b$ and $c'_b$ and all $X\in \Omega$, where $\Omega$ is a given set of $n \times p$ matrices. For example, in the model with missing data considered in the Introduction, this is the case if $\Omega= \{X: \ \max_{j=1,\dots,p} \frac1{n}\sum_{i=1}^n X_{ij}^4 \le m_4\}$ for some finite constant $m_4$, see \cite{RT2} and Section \ref{Sec5sm} in the Supplementary Material. 
Fix positive constants $\varepsilon, c, \lambda, m_2$ and denote by ${\cal P}$ the class of probability distributions ${\bf P}_X$ on the set of $n \times p$ matrices $X$ such that
\begin{equation}\label{ass_PX}
{\bf P}_X\Big[\kappa_q(s,1+\lambda)\ge cs^{-1/q},    \ \   \max_{j=1,\dots,p} \frac1{n}\sum_{i=1}^n X_{ij}^2 \le m_2, \ \ X\in \Omega \Big] \ge 1-\varepsilon.
\end{equation}
\begin{cor}\label{cor_PX}
Let $X$ be a random matrix with distribution ${\bf P}_X\in {\cal P}$, and let the above assumptions on the conditional distribution of $(\xi, W)$ given $X$ hold.   Assume that the true parameter $\t^* \in \Th$ is $s$-sparse, (\ref{hyps}) holds and set $\mu$ and $\tau$ as in (\ref{DefLambdaMu}). Then,  \eqref{11a} and \eqref{11c} hold with probability at least $1-9\e$.
\end{cor}


Although we do not pursue it here, Theorem \ref{th:conic} implies results on the correct selection of the sparsity pattern via a thresholding procedure, in the same spirit as it is done in~\cite{Lounici}.

Importantly, the bound (\ref{11a}) shows that the conic programming estimator is optimal in a minimax sense. Indeed, we give in Section \ref{sec:lower} lower bounds for estimation errors which are in agreement with the upper bounds in (\ref{11a}). The conic programming  estimator $\hat\t$ achieves this rate
with a computationally feasible procedure and does not need the knowledge of the parameters $|\t^*|_1$, $|\t^*|_2$ or $s$.

Inspection of the proof reveals that if Condition (\ref{hyps}) does not hold, the conclusions of Theorem \ref{th:conic} remain valid provided $|\theta^*|_2$ is replaced by $|\theta^*|_1$ in the bounds, thus leading to results analogous to those for the compensated MU selector. The next theorem formally states that. Note that the assumptions are different and somewhat weaker than in Theorem \ref{th:conic}.  

\begin{theorem}\label{th:conicRelaxed}
Assume (A1)--(A4), and that the true parameter $\t^*$ is $s$-sparse and belongs to $\Th$. Let $\e>0$, $1\le q\le \infty$, and set $\mu$ and $\tau$ as in (\ref{DefLambdaMu}). Then, with probability at least $1-8\e$,
\begin{eqnarray}\label{Relax11a}
|\hat\t-\t^*|_q&\le& \frac{C}{\kappa_q(s,1+\lambda)} \sqrt{\frac{\log (c'p/\e)}{n}} \{(\lambda+\lambda^{-1})|\theta^*|_1 +1\},
\end{eqnarray}
for some constants $C>0$ and $c'>0$. Under the same assumptions, the prediction risk admits the following bound,
with probability at least $1-8\e$:
\begin{eqnarray}\label{Relax11c}
\fracn\big|X(\hat\t-\t^*)\big|_2^2&\le& \frac{C}{\kappa_{\rm pr}^2(s,1+\lambda)}  \frac{\log (c'p/\e)}{n} \{(\lambda+\lambda^{-1})|\theta^*|_1 +1\}^2\,.
\end{eqnarray}
Furthermore, under no assumption on $X$, with the same probability, we have the following ``slow rate'' bound:
\begin{eqnarray}\label{Relax11d}
\fracn \big|X(\hat\t-\t^*)\big|_2^2&\le& C \sqrt{ \frac{\log (c'p/\e)}{n} }(\lambda^2 +\lambda^{-1})(|\theta^*|_1^2 +|\theta^*|_1).
\end{eqnarray}
The constants $C>0$ and $c'>0$ in \eqref{Relax11a}--\eqref{Relax11d} depend only on $m_2$ and on the constants appearing in Assumptions (A1)--(A4).
\end{theorem}


There are three different results in Theorem \ref{th:conicRelaxed}. The bound (\ref{Relax11a}) is based on the $\ell_q$-sensitivity measures without the sparsity condition (\ref{hyps}) and recovers the rates of the compensated MU selector. The second result (\ref{Relax11c}) presents a prediction rate but the prediction sensitivity allows for more general designs. 
Finally, the last result in Theorem \ref{th:conicRelaxed} provides a slow rate of convergence that requires no assumption on the design matrix.

\section{Minimax lower bounds for arbitrary estimators}
\label{sec:lower}

In this section, we show that the rates of convergence
obtained in Theorem \ref{th:conic} are optimal (up to
a logarithmic in $s$ term) in a minimax sense  for all estimators over the intersection of the class of $s$-sparse vectors and the $\ell_2$-sphere, respectively
$$
B_0(s)= \{\t: |\t|_0 \le s\} \ \ \ \mbox{and} \ \ \ S_2(R)= \{\t: |\t|_2 = R\},
$$
where $R>0$. Defining the parameter set as the intersection
$\Theta = B_0(s) \cap S_2(R)$ is motivated by the presence of both $s$ and $|\t^*|_2$ in the upper bounds of Theorem \ref{th:conic}.
Note that considering a deterministic $X$ means that $X$ is a nuisance parameter
of the model.
Thus, in the definition of the minimax risk, one should take the maximum not only over $\Theta$ but also over a
class of possible matrices~$X$. More generally, one can deal with random $X$ and with the maximum over a class of distributions
of $X$. We shall follow this approach with the class of distributions ${\cal P}$ introduced in Section~\ref{sec:conic}. The result of Corollary~\ref{cor_PX} corresponding to \eqref{11a} can be written as
\begin{equation}\label{lower1}
\sup_{{\bf P}_X\in \cal P} \ \sup_{\t\in B_0(s)\cap S_2(R) } \bP_{X,\theta} \Big[|\hat \t-\t|_q\ge C s^{1/q} \sqrt{\frac{\log(c'p/\e)}{n}} (R +1) \Big]\le 9\e,
\end{equation}
where, for $\theta\in \R^p$, we denote by $\mathbb P_{X,\theta}$ the probability
measure of the pair $(y, Z)$ satisfying (\ref{model1}). Our aim now is to prove the reverse inequality to \eqref{lower1} valid for all estimators.  For this purpose, instead of the maximum over all ${\bf P}_X\in \cal P$, it suffices in principle to consider  one particular distribution ${\bf P}_X$.  We choose it to be the distribution of Gaussian matrix $X$ with i.i.d. rows and positive definite covariance matrix. This enables us, in addition, to obtain minimax optimality when restricting the class ${\cal P}$ to one such Gaussian distribution only. With high probability, these matrices satisfy the RE condition, which implies the inequality $\kappa_q(s,1+\lambda)\ge cs^{-1/q}$ under the probability in \eqref{ass_PX} (see  for example \cite{RZ13} for details).
Also, we shall assume that  $\xi$ and $W$ are Gaussian with i.i.d. entries. In summary, we make the following assumption.

{\it
\begin{itemize}
\item[(A5)] The elements of the triplet $(\xi, X, W)$ are jointly independent. The components of $\xi$ and $W$ are i.i.d. Gaussian zero-mean random variables
with positive variances $\sigma^2$ and $\sigma^2_*$ respectively. The rows of $X$ are i.i.d. Gaussian zero-mean random vectors with covariance matrix $\Sigma>0$.
\end{itemize}
}

Denote by $\lambda_{\min}^\Sigma$ and $\lambda_{\max}^\Sigma$ the smallest and largest eigenvalues of $\Sigma$. The next theorem, proved in Appendix C, provides the desired minimax lower bound.

\begin{theorem}\label{th:lower}
Let $p\ge2$, $1\le q\le \infty$, $2\le s\le p$,  and $R>0$. Let Assumption (A5) hold and $s\log(p/s)/n \le  \bar c R^2/(R^2+1)$ for some constant $\bar c>0$. Then there exist constants $c>0$ and $c'>0$, depending only on $q, \sigma^2, \sigma_*^2, \bar c, \lambda_{\min}^\Sigma$ and $\lambda_{\max}^\Sigma$, such that
\begin{equation}\label{eq:th:lower}
\inf_{\hat T} \sup_{\t\in B_0(s)\cap S_2(R) } \bP_{X,\theta} \Big[|\hat T-\t|_q\ge c s^{1/q} \sqrt{\frac{\log(p/s)}{n}} (R+1) \Big]>c',
\end{equation}
where $\displaystyle \inf_{\hat T}$ denotes the infimum over all estimators, and we set $s^{1/\infty}=1$.
\end{theorem}


\section{Monte Carlo study}\label{sec:simulations}

In this section, we briefly illustrate the empirical performance of the estimators discussed above.
We consider the proposed conic programming estimator with $\lambda=0.5, \ 0.75 \mbox{and} \  1$ (denoted as Conic ($\lambda$)) and the Compensated MU selector (CompMU). To have benchmarks, we also compute the (unfeasible) Dantzig selector which knows $X$ (Dantzig X), and the Dantzig selector that uses only $Z$ (Dantzig Z), ignoring the errors-in-variables issue.

The simulation study uses the following data generating process
$$ y_i = x_i^T\theta^* + \xi_i, \ \ \ \ z_i = x_i + w_i.$$
Here, $\xi_i, w_i, x_i$ are independent and $\xi_i \sim {\cal N}(0,\sigma^2)$, $w_{i}\sim {\cal N}(0,\sigma_*^2I_{p\times p})$, $x_i\sim {\cal N}(0,\Sigma)$ where $I_{p\times p}$ is the identity matrix and $\Sigma$ is a $p\times p$ matrix with elements $\Sigma_{ij}=\rho^{|i-j|}$.
We set $\sigma = 0.128$, $\sigma_*=0.45$, and $\rho=0.25$.
For simplicity, we assume that $\sigma_*$ and $\sigma$ are known and we set $\hat D = D = \sigma_*^2I_{p\times p}$. The penalty parameters are set as $\tau = \mu = \sigma\sqrt{\log(p/\varepsilon)/n}$ for $\varepsilon = 0.05$.
 We consider two choices for the vector of unknown parameters $\theta^*$. The first choice is $\theta^* = (1,1,1,1,1,0,\ldots,0)^T $, which captures the case where the coefficients are well separated from zero. The second choice is $\theta^* = (1,1/2,1/3,1/4,1/5,0,\ldots,0)^T $, which represents the situation where $\theta^*$ is sparse with components that are not necessarily well separated from zero. The results are based on 100 replications. The implementation of these estimators was based on interior point methods which might not be suitable for large instances. The average running times of the six estimators were within a factor of two. For $p=10, 50$ the average running time of all six estimators was below one second, for $p=100$ all estimators were below 3 seconds and for $p=500$ the average running times were between 40 and 100 seconds. We note that Conic(1) and CompMU had very comparable running times, both DantzigX and DantzigZ had faster running times than the other estimators, while Conic(0.5) and Conic(0.75) requires slightly more time to be computed.

\begin{table}[h!]
\begin{center}
 \begin{tabular}{|l|ccc|ccc|ccc}
 \hline
First $\theta^*$ & \multicolumn{3}{|c|}{$n=300$ and $p=10$} &  \multicolumn{3}{|c|}{$n=300$ and $p=50$} \\
  \hline
   Method & Bias & RMSE  &  PR      & Bias  & RMSE  &  PR   \\
  \hline
  Conic (0.5)  & 0.0838151   & 0.1846383 &     0.1710643  & 0.0955776   & 0.2245046 &     0.2170111\\
  Conic (0.75)  &  0.0838151   & 0.1846383 &     0.1710643 & 0.0953689   & 0.2250219 &     0.2176691\\
  Conic (1) & 0.0838151   & 0.1846383 &     0.1710643 &  0.0956858   & 0.2253614 &     0.2180705 \\
  CompMU &  0.1566904   & 0.2191588      & 0.2225818 & 0.1840462   & 0.2362394      & 0.2507162 \\
  \hline
  Dantzig X & 0.0265486    &  0.0321528      & 0.0349530 & 0.0301636    &  0.0349420      & 0.0386731\\
  Dantzig Z &  0.2952845    &  0.3300527      & 0.3645317  &0.3078976    &  0.4166192      & 0.4174840 \\
  \hline
 \end{tabular}

~\\



\begin{tabular}{|l|ccc|ccc|ccc}
 \hline
First $\theta^*$ &  \multicolumn{3}{|c|}{$n=300$ and $p=100$} & \multicolumn{3}{|c|}{$n=300$ and $p=500$}\\
  \hline
   Method   & Bias  & RMSE  &  PR     & Bias & RMSE  &  PR\\
  \hline
  Conic (0.5)  &  0.1101178   & 0.2556778 &     0.2474407  & 0.1668239   & 0.2656095 &     0.263529 \\
  Conic (0.75)  &  0.0943678   & 0.2711839 &     0.2606997  &  0.1425789   & 0.2846916 &     0.2789745  \\
  Conic (1)  & 0.0942906   & 0.2734750 &     0.2631424  &  0.1276741   & 0.3121221 &     0.3093194 \\
  CompMU &  0.1910509   & 0.2539411      & 0.2658907 &  0.2052520   & 0.2657204      & 0.2772154  \\
  \hline
  Dantzig X &  0.0317776    &  0.0366155      & 0.0403419 &  0.0352309    &  0.0403134      & 0.0448000\\
  Dantzig Z &  0.3081669    &  0.4994041      & 0.4652972 & 0.3536668    &  0.6865989      & 0.5921541   \\
\hline
  \end{tabular}
  \caption{\footnotesize Simulation results for the first choice of $\theta^*$. For each estimator we provide average bias (Bias), average root-mean squared error (RMSE), and average prediction risk (PR).}\label{Table:MC}
\end{center}
\end{table}

Table \ref{Table:MC} reports the simulation results in the case $\theta^* = (1,1,1,1,1,0,\ldots,0)^T $. As expected, the performance of all the estimators deteriorates as $p$ grows but only slightly. Also, the (unfeasible) estimator based on Dantzig selector that observes $X$ outperforms all feasible options. The estimator that ignores the errors-in-variables issue appears with a higher bias leading to the worse result in terms of root mean square error and empirical risk. The performance of the feasible estimators discussed in this paper is between these two benchmarks. The three conic estimators exhibit a better result than the compensated MU selector when $p=10, 50$. For the larger dimensions $p=100, 500$, their performance becomes similar to that of the compensated MU selector. Nonetheless, the conic estimator with $\lambda =0.5$ is slightly better than all the other feasible estimators.
We also note that for the small dimension $p=10$, all three conic estimators give the same results. The reason is that the conic constraint was not active in the simulations for $p=10$ so that the estimators were the same for the range of $\lambda$ under consideration. This was not the case for $p=50,100,500$. These findings are very much aligned with the theoretical properties of each estimator and sustain that the impact of errors-in-variables can be substantial.

Table \ref{Table:MC2} reports the results for the second choice of $\theta^*$, where the coefficients are not well separated from zero. They are qualitatively the same as before. The results confirms the robustness of the conclusions with respect to possible model selection errors which are unavoidable when the coefficients are not well separated from zero.

\begin{table}[h!]
\begin{center}
 \begin{tabular}{|l|ccc|ccc|ccc}
 \hline
Second $\theta^*$ & \multicolumn{3}{|c|}{$n=300$ and $p=10$} &  \multicolumn{3}{|c|}{$n=300$ and $p=50$} \\
  \hline
   Method & Bias & RMSE  &  PR      & Bias  & RMSE  &  PR   \\
  \hline
  Conic (0.5)  &   0.0564816   & 0.1020763 &     0.0987642  &  0.0684162   & 0.1236275 &     0.1223037\\
  Conic (0.75) & 0.0564816   & 0.1020763 &     0.0987642 &  0.0682720   & 0.1229000 &     0.1219398 \\
  Conic (1) &  0.0564816   & 0.1020763 &     0.0987642 &  0.0683291   & 0.1227749 &     0.1218898 \\
  CompMU &  0.0839431   & 0.1171633      & 0.1204765 &   0.1007774   & 0.1318303      & 0.1396494\\
  \hline
  Dantzig X &  0.0265486    &  0.0321528      & 0.0349530 &  0.0301636    &  0.0349420      & 0.0386731 \\
  Dantzig Z &  0.1885828    &  0.2024266      & 0.2138763 & 0.1949159    &  0.2314964      & 0.2319208 \\
  \hline
 \end{tabular}
%

~\\

\begin{tabular}{|l|ccc|ccc|ccc}
 \hline
Second $\theta^*$  &  \multicolumn{3}{|c|}{$n=300$ and $p=100$} & \multicolumn{3}{|c|}{$n=300$ and $p=500$}\\
  \hline
   Method   & Bias  & RMSE  &  PR     & Bias & RMSE  &  PR\\
  \hline
  Conic (0.5)  &  0.0714621   & 0.1374637 &     0.1349551 & 0.0945558   & 0.1472914 &     0.1477198 \\
  Conic (0.75)  &  0.0713670   & 0.1378301 &     0.1353203 &  0.0824510   & 0.1589884 &     0.1565416 \\
  Conic (1)  & 0.0716242   & 0.1381810 &     0.1357000 & 0.0783823   & 0.1682841 &     0.1658849 \\
  CompMU & 0.1063728   & 0.1405472      & 0.1479579  & 0.1131005 & 0.1477336      & 0.1545960 \\
  \hline
  Dantzig X &  0.0317776    &  0.0366155      & 0.0403419  & 0.0352309    &  0.0403134      & 0.0448000 \\
  Dantzig Z & 0.1978958    &  0.2536633      & 0.2432222 & 0.2152972    &  0.3145349      & 0.2766815 \\
\hline
  \end{tabular}
  \caption{\footnotesize Simulation results for the second choice of $\theta^*$. For each estimator we provide average bias (Bias), average root-mean squared error (RMSE), and average prediction risk (PR).}\label{Table:MC2}
\end{center}
\end{table}


\noindent {\bf Acknowledgement}
We would like to thank
Anatoli Juditsky for helpful remarks.
This work is supported by GENES and by the French National Research Agency (ANR) as part of Idex Grant ANR -11- IDEX-0003-02, Labex ECODEC (ANR - 11-LABEX-0047), and IPANEMA grant (ANR-13-BSH1-0004-02), and by the ``Chaire Economie et Gestion des Nouvelles Donn\'ees", under the
auspices of Institut Louis Bachelier, Havas-Media and Paris-Dauphine.

\section*{Appendix A. Bounds on stochastic error terms}\label{appendixA}

In this appendix, we give upper bounds on the stochastic error terms appearing in the main results. In what follows, $D$ is the diagonal matrix with diagonal elements
$\sigma_j^2$, $j=1,\dots,p$, and for a square matrix $A$, we denote
by $\text{Diag}\{A\}$ the matrix with the same dimensions as $A$,
the same diagonal elements as $A$ and all off-diagonal elements
equal to zero. The following lemma is proved in \cite{RT2}.

\begin{lemma}\label{lem2} Let $0<\e<1$ and assume (A1)-(A3). Then, with probability at least $1-\varepsilon$,
\begin{eqnarray*}
&&\big|\fracn X^TW\big|_\infty \le \delta_1(\varepsilon),\quad \big| \fracn X^T\xi\big|_\infty \le \delta_2(\varepsilon),\quad \big| \fracn W^T\xi\big|_\infty \le \delta_3(\varepsilon),\\
&& \big|\fracn(W^TW-\text{\rm Diag}\{W^TW\})\big|_\infty \le \delta_4(\varepsilon),\quad\big|\fracn \text{\rm Diag}\{W^TW\}-D\big|_\infty \le \delta_5(\varepsilon),
\end{eqnarray*}
where $\delta_1(\varepsilon)=\tilde \delta(\varepsilon,\sigma_*,2p^2)$, $\delta_2(\varepsilon)=\tilde \delta(\varepsilon,\sigma,2p),$ $\delta_3(\varepsilon)=\delta_5(\varepsilon)=\bar \delta
(\varepsilon,2p)$, $\delta_4(\varepsilon)=\bar \delta
(\varepsilon,p(p-1))$
and for an integer $N$,
$$
\tilde \delta(\varepsilon,a,N) = a
\sqrt{\frac{2m_2\log(N/\varepsilon)}{n}}, \quad \bar \delta (\varepsilon,N) = \max\left(\gamma_0
\sqrt{\frac{2\log(N/\varepsilon)}{n}},\
\frac{2\log(N/\varepsilon)}{t_0n}\right),
$$
with $\gamma_0, t_0$ are positive constants depending only on $\sigma$ and $\sigma_*$.
\end{lemma}
\noindent We now give the second lemma.
\begin{lemma}\label{lem3a}
Let $0<\e<1$, $\theta^*\in \R^p$ and assume (A1)-(A3). Then, with probability at least $1-\varepsilon$,
\begin{align}
&\big|\fracn X^TW\theta^*\big|_{\infty} \leq \delta_1'(\e)|\theta^*|_2,\label{lem3a1}
\end{align}
where $\delta_1'(\e)= \sigma_{*} \sqrt{\frac{2m_2\log(2p/\e)}{n}}$.
In addition, 
with probability at least $1-\varepsilon$,
\begin{align}
&\big|\fracn (W^TW-{\rm Diag}\{W^TW\})\theta^*\big|_{\infty}\leq\delta_4'(\e)|\theta^*|_2,\label{lem3a2}
\end{align}
where 
$$
\delta_4'(\e) =\max\left(\gamma_2
\sqrt{\frac{2\log(2p/\varepsilon)}{n}},\
\frac{2\log(2p/\varepsilon)}{t_2n}\right),
$$
and $\gamma_2, t_2$ are positive constants depending only on $\sigma_*$.
\end{lemma}

\begin{proof}  If $\theta^*=0$, the result is obvious. So we assume that $\theta^*\ne 0$. Let $v=\theta^*/|\theta^*|_2$. We can write
\begin{align}
&
\big|\fracn X^TW\theta^*\big|_{\infty} = |\theta^*|_2 \max_{j=1,\dots,p}\left|\frac{1}{n}\sum_{i=1}^n X_{ij}(w_i,v)\right|,
\label{18a}
\end{align}
where $(w_i,v)=\sum_{k=1}^pW_{ik}v_k$ and we denote by $W_{ik}$ and $v_k$ the elements of the matrix $W$ and the vector $v$ respectively. By Assumption (A3), the random variable $(w_i,v)$ is subgaussian with variance parameter $\sigma_*^2$. Using this together with the independence of the $w_i$ for different $i$, we get that for all $t\in \R$,
{\small \begin{align*}
\E \Big[\exp \Big(\frac{t}{n}\sum_{i=1}^n X_{ij}(w_i,v)\Big)\Big]&=\prod_{i=1}^n\E \Big[\exp \Big(\frac{t}{n} X_{ij}(w_i,v)\Big)\Big] \le \prod_{i=1}^n \exp \Big(\frac{\sigma_*^2t^2X_{ij}^2}{2n^2}\Big)\leq \exp \Big(\frac{\sigma_*^2m_2t^2}{2n}\Big).
\end{align*} }
Thus, the random variable $$\eta_j=\frac{1}{n}\sum_{i=1}^n X_{ij}(w_i,v)$$ is $\gamma_1$-subgaussian with $\gamma_1=\sigma_*\sqrt{m_2/n}$. This implies the classical tail bound $$\mathbb{P}[|\eta_j|\ge \delta] \le 2 \exp \big(-\delta^2/(2\gamma_1^2)\big),$$ for any $\delta>0$. This together with (\ref{18a}) and the union bound yields (\ref{lem3a1}).\\

\noindent To prove (\ref{lem3a2}),
we shall use the following fact (see for example Lemma 5.14 in \cite{vershynin}): If $\eta$ is a subgaussian random variable with variance parameter $\gamma^2$, then $\eta^2$ is sub-exponential, that is there exist  constants $\gamma_0=\gamma_0(\gamma)$ and $t_0=t_0(\gamma)$ such that
\begin{equation}\label{subexp}
\E[\text{exp}(t\eta^2)]\leq
\text{exp}(\gamma_0^2t^2/2), \quad |t|\le t_0.
\end{equation}
Analogously to (\ref{18a}), we obtain
\begin{align}
&
\big|\fracn(W^TW-{\rm Diag}\{W^TW\})\theta^*\big|_{\infty} = |\theta^*|_2 \max_{j=1,\dots,p}|\eta'_j|,
\label{18b}
\end{align}
where $$\eta'_j=\frac{1}{n}\sum_{i=1}^n W_{ij}\sum_{k=1, k\ne j}^pW_{ik}v_k.$$ Now, for all $t\in \R$, we have
{\small \begin{align*}
\E[\text{exp}(t\eta'_j)]&= \prod_{i=1}^n \E\Big[\exp\Big(\frac{tW_{ij}}{n}\sum_{k=1, k\ne j}^pW_{ik}v_k \Big)\Big] \le \prod_{i=1}^n\E\Big[\exp\Big\{\frac{t}{2n}\Big(W_{ij}^2+\big(\sum_{k=1, k\ne j}^pW_{ik}v_k\big)^2\Big)\Big\}\Big].
\end{align*}}
Then, using Cauchy-Schwarz inequality, we get
$$\E[\text{exp}(t\eta'_j)]\le \prod_{i=1}^n\Big\{\E\Big[\exp\Big(\frac{tW_{ij}^2}{n}\Big)\Big] \E\Big[\exp\Big(\frac{t}{n}
\big(\sum_{k=1, k\ne j}^pW_{ik}v_k\big)^2\Big)\Big]\Big\}^{1/2}.$$
Recall that Assumption (A3) implies that both $W_{ij}$ and $\sum_{k=1, k\ne j}^pW_{ik}v_k$ are $\sigma_*$-subgaussian. Consequently, in view of (\ref{subexp}), their squared values are $(\gamma_0(\sigma_*), t_0(\sigma_*))$-sub-exponential, which yields
\begin{align*}
\E[\text{exp}(t\eta'_j)]
&\le \prod_{i=1}^n\exp\Big(\frac{\gamma_0(\sigma_*)^2}{2}\Big(\frac{t}{n}\Big)^2\Big)=\exp\Big(\frac{\gamma_0(\sigma_*)^2t^2}{2n}\Big), \quad |t|\le t_0(\sigma_*)n.
\end{align*}
Set $\gamma_2=\gamma_0(\sigma_*)$ and $t_2 =t_0(\sigma_*).$ The last display states that $\eta'_j$ is $(\gamma_2/\sqrt{n}, t_2 n)$-sub-exponential. This implies the tail bound
 $$\mathbb{P}( |\eta'_j|\ge \delta) \le 2 \max\big(\exp(-n\delta^2/(2\gamma_2^2)), \exp (-\delta t_2n/2)\big),$$ for any $\delta>0$. This together with (\ref{18b}) and the union bound yields (\ref{lem3a2}).
\end{proof}


\section*{Appendix B. Proofs of Theorem \ref{th:conic} and Theorem \ref{th:conicRelaxed}}\label{appendixB}

Set for brevity $\delta_i=\delta_i(\e)$, $\delta_i'=\delta_i'(\e)$, $b=b(\e)$. We first prove some preliminary lemmas.

\begin{lemma}\label{lem3}
Assume (A1)--(A4). Then with probability at least $1-6\varepsilon$, the pair $(\theta, t)=(\theta^*, |\theta^*|_2) $ belongs to the feasible set of the minimisation problem
(\ref{conic}).
\end{lemma}

\begin{proof} First, note that
$Z^T(y-Z\theta^*)+n{\widehat D}\theta^*$ is equal to
\begin{align*}
&-X^TW\theta^*+X^T\xi+W^T\xi-(W^TW-\text{Diag}\{W^TW\})\theta^*\\
&-(\text{Diag}\{W^TW\}-nD)\theta^*+n(\widehat D-D)\theta^*.
\end{align*}
By definition of $\delta_i$ and $b$, with probability at least
$1-4\varepsilon$, we have
\begin{align}
&|\fracn X^T\xi|_{\infty}+|\fracn W^T\xi|_{\infty}\leq \delta_2+\delta_3\label{lem3.2}\\
&|(\fracn \text{Diag}\{W^TW\}-D)\theta^*|_{\infty}\leq|\fracn \text{Diag}\{W^TW\}-D|_{\infty}|\theta^*|_\infty\leq \delta_5|\theta^*|_2\label{lem3.3}\\
&|(\widehat D-D)\theta^*|_{\infty}\leq b|\theta^*|_\infty\leq b |\theta^*|_2,\label{lem3.4}
\end{align}
where in (\ref{lem3.3}) and (\ref{lem3.4}) we have used that the considered matrices are diagonal. Also, by Lemma \ref{lem3a}, with probability at least
$1-2\varepsilon$, we have
\begin{align}
&|\fracn X^TW\theta^*|_{\infty}\leq \delta_1'|\theta^*|_2\label{ineg}\\
&|\fracn(W^TW-\text{Diag}\{W^TW\})\theta^*|_{\infty}\leq\delta_4'|\theta^*|_2.\label{ineg1}
\end{align}
Combining the decomposition of $Z^T(y-Z\theta^*)+n{\widehat D}\theta^*$ together with (\ref{lem3.2})-(\ref{ineg1}), we find that $$\big|\fracn Z^T(y-Z\theta^*)+\widehat{D}\theta^*\big|_\infty\leq\mu |\theta^*|_2
+\tau,$$ with probability at
least $1-6\varepsilon$, which implies the lemma.
\end{proof}
$~$\\

\noindent We now give two lemmas which will be crucial in the proof of our main theorem on the accuracy of the conic programming based estimator (Theorem \ref{th:conic}).
\begin{lemma}\label{lem4} Assume (A1)--(A4). Let $J=\{j: \theta^*_j\ne 0\}$. Then
with probability at least $1-6\varepsilon$ (on the same event as in Lemma \ref{lem3}), we have
\begin{align}
& |(\hat\theta-\theta^*)_{J^c}|_1\leq (1+\lambda)|(\hat \theta-\theta^*)_{J}|_1, \label{c1}\\
&  \hat t \le (1/\lambda) |\hat \theta-\theta^*|_1 + |\theta^*|_2.\label{c2}
\end{align}
\end{lemma}

\begin{proof} Set $\Delta=\hat\theta-\theta^*$. On the event of Lemma \ref{lem3}, $(\theta^*, |\theta^*|_2)$ belongs to the feasible set of the minimisation problem (\ref{conic}). Consequently,
\begin{align}
& |\hat\theta|_1+ \lambda|\hat\theta|_2\leq |\hat\theta|_1 + \lambda \hat t \le |\theta^*|_1 + \lambda |\theta^*|_2. \label{c3}
\end{align}
This implies $
|\Delta_{J^c}|_1 \le |\Delta_J|_1 +  \lambda (|\theta^*|_2-|\hat\theta|_2) \le |\Delta_J|_1 + \lambda |\Delta_J|_2\le (1+\lambda)|\Delta_J|_1,
$
and (\ref{c1}) follows. To prove (\ref{c2}), it suffices to note that (\ref{c3}) implies
\begin{align*}
& \lambda \hat t \le |\theta^*|_1 -|\hat\theta|_1  + \lambda |\theta^*|_2 \le |\hat \theta-\theta^*|_1 + \lambda |\theta^*|_2.
\end{align*}
\end{proof}


\begin{lemma}\label{lem5} Assume (A1)--(A4). Then,
on a subset of the event of Lemma \ref{lem3} having probability at least $1-8\varepsilon$, we have
\begin{align}
& \big|\fracn X^TX(\hat\theta-\theta^*)\big|_\infty \le \mu_1 |\theta^*|_2 +  \mu_2 |\hat\theta-\theta^*|_1+ \tau_1, \label{lem5.1}
\end{align}
where $\mu_1=\mu+b+\delta_1'+\delta_4'+\delta_5$, $\mu_2=\mu/\lambda+b+2\delta_1+\delta_4+\delta_5$ and $\tau_1=\tau +\delta_2+\delta_3$.
\end{lemma}
%
\begin{proof}
Throughout the proof, we assume that we are on the event of probability
at least $1-6\varepsilon$ where Inequalities (\ref{lem3.2}) -- (\ref{ineg1}) hold and $(\theta^*, |\theta^*|_2)$ belongs to the feasible set of the minimisation problem (\ref{conic}). We
have that
$|\fracn X^TX\Delta|_{\infty}$
is smaller than
\begin{align*} &|\fracn Z^T(Z\hat{\theta}-y)-{\widehat D}\hat{\theta}|_{\infty}+|(\fracn Z^TW-D)\hat{\theta}|_{\infty}+|(\widehat
D-D)\hat{\theta}|_{\infty}+|\fracn Z^T\xi|_{\infty}+|\fracn W^TX\Delta|_{\infty}.
\end{align*}
Using the fact that $(\hat\theta, \hat t)$ belongs to the feasible set of the minimisation problem (\ref{conic}) together with (\ref{c2}), we obtain
$$
 |\fracn Z^T(Z\hat{\theta}-y)-{\widehat D}\hat{\theta}|_{\infty} \le \mu  \hat t + \tau \le (\mu/\lambda) |\hat \theta-\theta^*|_1 + \mu |\theta^*|_2 + \tau.
$$
Therefore,
$|\fracn X^TX\Delta|_{\infty}$ does not exceed
$$(\mu/\lambda) |\hat \theta-\theta^*|_1 + \mu |\theta^*|_2 + \tau_1
+|(\fracn Z^TW-D)\hat{\theta}|_{\infty}+|(\widehat
D-D)\hat{\theta}|_{\infty}+|\fracn W^TX\Delta|_{\infty}.$$
We now bound the last expression using the fact that $\hat\theta=\theta^*+\Delta$ together with Assumption (A4) and (\ref{lem3.4}). This gives
\begin{align}\label{eq:*}
|\fracn X^TX\Delta|_{\infty}&\le ((\mu/\lambda) +b)|\Delta|_1 + (\mu +b) |\theta^*|_2 + \tau_1+|(\fracn Z^TW-D)\theta^*|_{\infty}\\& + |(\fracn Z^TW-D)\Delta |_{\infty}+|\fracn W^TX\Delta|_{\infty}.\nonumber
\end{align}
Remark that
\begin{multline*}|(\fracn Z^TW-D)\Delta |_{\infty}\leq
|\fracn Z^TW-D|_{\infty}|\Delta|_1\\
\le
\big(|\fracn(W^TW-\text{Diag}\{W^TW\})|_{\infty}
+|\fracn\text{Diag}\{W^TW\}-D|_{\infty}
+|\fracn X^TW|_{\infty}\big)|\Delta|_1.
\end{multline*}
Therefore,
\begin{equation} \label{c7}
|(\fracn Z^TW-D)\Delta |_{\infty}
\le (\delta_1+\delta_4+\delta_5)|\Delta|_1,
\end{equation}
with probability at least $1-8\e$ (since we intersect the initial event of probability at least $1-6\e$ with the event of probability at least $1-2\e$ where the bounds $\delta_1$ and $\delta_4$ hold for the corresponding terms).  Next, on the same event of probability at least $1-8\e$,
\begin{eqnarray}
 |\fracn W^TX\Delta|_{\infty}\leq |\fracn X^TW|_{\infty} |\Delta|_1 \leq \delta_1 |\Delta|_1.
\label{c8}
\end{eqnarray}
Finally, in view of Lemma \ref{lem3a} and (\ref{lem3.3}), on the initial event of probability at least $1-6\e$,
\begin{align}
&|(\fracn Z^TW-D)\theta^* |_{\infty}\nonumber\\
\leq&
|\fracn(W^TW-\text{Diag}\{W^TW\}) \theta^*|_{\infty}
+|(\fracn\text{Diag}\{W^TW\}-D)\theta^*|_{\infty} + \ |\fracn X^TW \theta^* |_{\infty}\nonumber\\
\leq& (\delta_1'+\delta_4'+\delta_5)|\theta^*|_2. \label{c9}
\end{align}
To complete the proof, it suffices to plug (\ref{c7}) -- (\ref{c9}) in \eqref{eq:*} and to set  $\mu_1=\mu+b+\delta_1'+\delta_4'+\delta_5$ and $\mu_2=\mu/\lambda+b+2\delta_1+\delta_4+\delta_5$.
\end{proof}
$~$\\

We are ready to give the proof of Theorem \ref{th:conic}. Throughout the proof, we assume that we are on the event of probability
at least $1-8\varepsilon$ of Lemma \ref{lem5} where the results of Lemmas \ref{lem3}, \ref{lem4} and \ref{lem5} hold. Property (\ref{c1}) in Lemma \ref{lem4} implies that $\Delta$ is in the cone $C_J(1+\lambda)$. Therefore, by definition of $\ell_q$-sensitivity and Lemma~\ref{lem5}, we have
$$
\kappa_q(s,1+\lambda)|\Delta|_q\le \big|\fracn X^TX\Delta\big|_\infty \le \mu_1 |\theta^*|_2 +  \mu_2 |\Delta|_1+ \tau_1.
$$
Furthermore, using again (\ref{c1}), we have
\begin{eqnarray*}
|\Delta|_1 &=& |\Delta_{J^c}|_1+|\Delta_J|_1\le (2+\lambda)|\Delta_J|_1\\
&\le & (2+\lambda)s^{1-1/q}|\Delta_{J}|_q \le (2+\lambda) s^{1-1/q}|\Delta|_q.
\end{eqnarray*}
It follows that
$$
(\kappa_q(s,1+\lambda)- (2+\lambda)\mu_2 s^{1-1/q})|\Delta|_q \le  \mu_1 |\theta^*|_2 + \tau_1.
$$
Further,  the assumption $s\le c_1(\lambda^{-1}+\lambda)^{-1}\sqrt{n/\log (p/\e)} $  implies
$$
\{c- \mu_2 c_1 (2+\lambda)(\lambda^{-1}+\lambda)^{-1}\sqrt{n/\log (p/\e)}\} s^{-1/q}|\Delta|_q \le  \mu_1 |\theta^*|_2 + \tau_1.
$$ Recall that $\mu_2 \le (1 + \lambda^{-1})a \sqrt{\log (p/\e)/n}$, where $a>0$ is a constant.
Therefore, (\ref{11a}) follows if $(2+\lambda)(1 + \lambda^{-1})(\lambda^{-1}+\lambda)^{-1}c_1 a<c/2$. Since $(2+\lambda)(1 + \lambda^{-1})(\lambda^{-1}+\lambda)^{-1}\le 5$ we have that $c_1 < c/(10a)$ yields (\ref{11a}).
To prove (\ref{11c}), write first
$$
\fracn \left|X\Delta\right|_2^2 \le \fracn\left|X^TX\Delta\right|_\infty|\Delta|_1.
$$
Next remark that from (\ref{11a}), we have
\begin{equation}\label{eq:**}
|\Delta|_1\leq C s \sqrt{\frac{\log (c'p/\e)}{n}} (|\theta^*|_2 +1)
\end{equation}
and recall that from Lemma \ref{lem5},
\begin{equation}\label{sb}
\big|\fracn X^TX\Delta\big|_\infty \le C\sqrt{\frac{\log (c'p/\e)}{n}}(1+|\theta^*|_2+|\Delta|_1).
\end{equation}
Furthermore, from \eqref{eq:**} and \eqref{hyps} we also have
$|\Delta|_1\leq C(|\theta^*|_2 +1),$ for some constant $C>0$. Then (\ref{11c}) is easily deduced. This ends the proof of Theorem \ref{th:conic}.\\

We now give the proof of Theorem \ref{th:conicRelaxed}. We place ourselves in the same framework as in the proof of Theorem \ref{th:conic}. By the definition of the estimator, $|\Delta|_1\leq |\hat \t|_1 + |\theta^*|_1 \leq (|\theta^*|_1 +\lambda|\theta^*|_2) + |\theta^*|_1 \leq (2+\lambda)|\theta^*|_1$, where we have used that $|\theta^*|_2 \leq |\theta^*|_1$. This and Lemma \ref{lem5} yield
$$
 \big|\fracn X^TX\Delta\big|_\infty \le \mu_1 |\theta^*|_2 +  \mu_2 |\Delta|_1+ \tau_1 \leq (\mu_1+(2+\lambda)\mu_2) |\theta^*|_1 + \tau_1.
$$
Therefore, arguing as in the proof of Theorem \ref{th:conic}, we find
$$
\kappa_q(s,1+\lambda)|\Delta|_q \le (\mu_1+(2+\lambda)\mu_2) |\theta^*|_1 + \tau_1,
$$
which implies  (\ref{Relax11a}) since $\mu_2 \le (1+\lambda^{-1})a\sqrt{\log(p/\varepsilon)/n}$ for some constant $a>0$.
To prove (\ref{Relax11c}), we note that by definition of $\kappa_{\rm pr}$ and the fact that $\Delta\in C_J(1+\lambda)$,
$$
\frac{\kappa_{\rm pr}^2(s,1+\lambda)}{n}\left|X\Delta\right|_2^2 \le \fracn\left|X^TX\Delta\right|_\infty^2 \leq \left\{ (\mu_1+ (2+\lambda)\mu_2) |\theta^*|_1 + \tau_1 \right\}^2.
$$
Finally, since $|\Delta|_1\leq (2+\lambda)|\theta^*|_1$ and $\mu_2 \le (1 + \lambda^{-1})a \sqrt{\log (p/\e)/n}$, (\ref{Relax11d}) follows from
$$
\fracn\left|X\Delta\right|_2^2 \le \fracn\left|X^TX\Delta\right|_\infty|\Delta|_1\leq \{(\mu_1+(2+\lambda)\mu_2) |\theta^*|_1 + \tau_1 \} (2+\lambda)|\theta^*|_1.
$$
This concludes the proof of Theorem \ref{th:conicRelaxed}.
\section*{Appendix C. Proof of Theorem \ref{th:lower}}
Again, throughout, we  denote by $c$ a positive constant which may vary from line to line.
To derive the lower bounds, we apply Theorem~2.7 in \cite{tsy09}. Thus, we define a finite set of ``hypotheses" included
in $B_0(s)\cap S_2(R)$. To this end, we first introduce
$$\mathcal{M}=\big\{x\in\{0,1\}^{p-1}: \rho_H({\bf 0},x)=s-1\big\},$$ where $\rho_H$ denotes the Hamming distance between elements of $\{0,1\}^{p-1}$, and ${\bf 0}$ is the zero vector. Then,
there exists a subset $\mathcal{M}'$ of $\mathcal{M}$ such that for any $x,x'$ in $\mathcal{M}'$
with $x\neq x'$, we have $\rho_H(x,x')>s/16$, and moreover,
$$\text{log}|\mathcal{M}'|\geq c_1's\log\left(\frac{p}{s}\right)$$
for some absolute constant $c_1'>0$. Indeed, this follows from the Varshamov-Gilbert bound (see Lemma 2.9 in \cite{tsy09}) if $s-1>(p-1)/2$  and from Lemma A.3 in \cite{RigTsy11} if $s-1\le (p-1)/2$.

We denote by $\omega'_j$ the elements of the finite set~$\mathcal{M}'$. For $j= 1,\dots, |\mathcal{M}'|$, we define vectors $\omega_j\in \{0,1\}^{p}$ with components
$\omega_{j1}=0$ and  $\omega_{jk}=\omega'_{j(k-1)}$ for $k\geq 2$, where $\omega_{jk}$ is the $k$-th component of $\omega_{j}$. We also define
$\omega_{0}$ as the vector in $\{0,1\}^{p}$ with all components equal to $0$ except the first one equal to 1.

We now define the set of ``hypotheses" $({\bar \omega}_j, j=0,\dots, |\mathcal{M}'|+1)$, where ${\bar \omega}_0=R\omega_{0}$, and
$${\bar \omega}_j=\frac{R}{\sqrt{1+\gamma^2 (s-1)}}(\omega_{0}+\gamma \omega_{j}), \quad j=1,\dots, |\mathcal{M}'|+1.$$ Here, $\gamma
$ is a positive parameter to be defined. Note that the sparsity of
${\bar \omega}_j$ is equal to $s$ and that $|{\bar \omega}_j|_2=R$. Thus all ${\bar \omega}_j$ belong to $B_0(s)\cap S_2(R)$.
Let
$$\tilde\Sigma=\sigma^2_*(\Sigma+\sigma^2_*I_{p\times p})^{-1}\text{ and }\Gamma=I_{p\times p}-\tilde\Sigma.$$
For $\theta\in\mathbb{R}^p$, we set $c_{\theta}=\theta^T\Gamma\theta$ and we write ${\cal K}(\mathbb P,\mathbb Q)$ for the Kullback-Leibler divergence between two probability measures $\mathbb P$ and $\mathbb Q$. For $j\geq 1$, by Lemma~\ref{lem:kullback} in the Supplementary Material, we have
\begin{align*}
{\cal K} (\mathbb P_{{\bar \omega}_j}, \mathbb P_{{\bar \omega}_0})&\leq \frac{cn}{1+R^2}\Big(R^2\Big(\frac{\sqrt{1+\gamma^2 (s-1)}-1}{\sqrt{1+\gamma^2 (s-1)}}\Big)^2+\frac{R^2}{1+\gamma^2 (s-1)}\gamma^2 s+|c_{{\bar \omega}_j}-c_{{\bar \omega}_0}|\Big)
\\&
\leq \frac{cn}{1+R^2}\Big(\frac{\gamma^2 R^2 s}{1+\gamma^2 (s-1)}+|c_{{\bar \omega}_j}-c_{{\bar \omega}_0}|\Big).
\end{align*}
Now,
\begin{align*}
|c_{{\bar \omega}_j}-c_{{\bar \omega}_0}| &=|{{\bar \omega}_j}^T\Gamma {{\bar \omega}_j}-{{\bar \omega}_0}^T\Gamma {\bar \omega}_0|
=\left|{R^2{\omega}_0}^T\Gamma {\omega}_0 - \frac{R^2}{1+\gamma^2 (s-1)}{{\omega}_0}^T\Gamma {\omega}_0-
\frac{\gamma^2 R^2}{1+\gamma^2 (s-1)}{{\omega}_j}^T\Gamma {\omega}_j\right|
\\&
\le \frac{\gamma^2 R^2}{1+\gamma^2 (s-1)}\left((s-1){\omega}_0^T\Gamma {\omega}_0 + {{\omega}_j}^T\Gamma {\omega}_j\right)
\le \frac{2\lambda_{\max} \gamma^2 R^2 (s-1)}{1+\gamma^2 (s-1)},
\end{align*}
where $\lambda_{\max}$ denotes the largest eigenvalue of $\Gamma$ and the last inequality is due to the fact that $|{\omega}_0|_2=1, |{\omega}_j|_2^2=s-1$. Combining the last two displays yields
\begin{align*}
{\cal K} (\mathbb P_{{\bar \omega}_j}, \mathbb P_{{\bar \omega}_0})
\leq c_2'n \gamma^2 s \frac{R^2}{1+R^2},
\end{align*}
where $c_2'>0$ is a constant that does not depend on $s$, $p$, $R$, and $n$. Now, taking
\begin{equation}\label{eq:gamma}
\gamma=\Big(\frac{c_1'}{16 c_2' n}\frac{1+R^2}{R^2}\log\left(\frac{p}{s}\right)\Big)^{1/2},
\end{equation}
we obtain, for all $j$,
$${\cal K} (\mathbb P_{{\bar \omega}_j}, \mathbb P_{{\bar \omega}_0})\leq \frac{1}{16}\log |\mathcal{M}'|.$$
Next,
for $j$ and $j'$ both different from $0$,
$$|{\bar \omega}_j-{\bar \omega}_{j'}|_q=\frac{R\gamma}{\sqrt{1+\gamma^2 (s-1)}}\Big(\sum_{k=1}^{p-1}|\omega_{jk}-\omega_{j'k}|^q\Big)^{1/q}\geq cs^{1/q}\frac{R\gamma}{\sqrt{1+\gamma^2 (s-1)}}\,$$
and for $j\neq 0$,
$$|{\bar \omega}_j-{\bar \omega}_{0}|_q\ge \frac{ R\gamma |{\omega}_j|_q}{\sqrt{1+\gamma^2 (s-1)}} \geq cs^{1/q}\frac{R\gamma}{\sqrt{1+\gamma^2 (s-1)}}.$$
The definition of $\gamma$ in \eqref{eq:gamma} and the conditions in Theorem \ref{th:lower} imply that, for any $j$ and $j'$,
$$|{\bar \omega}_j-{\bar \omega}_{j'}|_q\geq cs^{1/q}R\gamma\geq cs^{1/q}(R+1)\sqrt{\frac{\text{log}(p/s)}{n}}\,.$$
We can now apply Theorem 2.7 in \cite{tsy09} to obtain the result.

\newpage

\section*{Supplementary Material}

\setcounter{section}{0}

\section{Computation of the compensated MU selector}
\label{sec:ComputationC}

The goal of this section is to show that the minimisation problem (\ref{C}) defining the compensated MU selector can be solved numerically in an efficient way. This algorithmic issue can be intricate since the problem is, in general, not convex, except in some specific situations. For example,
if $\Theta=(\mathbb{R}^+)^p$, it obviously reduces to linear programming. However, we shall see that under an additional mild technical hypothesis, solutions can be obtained using convex or even linear programming. It is therefore computationally simpler than the conic programming estimator $\hat\theta$. We focus here only on algorithmic aspects. Therefore, we do not recall the assumptions under which the problem admits a solution and the estimator enjoys relevant properties. We refer to \cite{RT2} where these issues are addressed in detail.

For brevity, we write $$S(\theta)= \fracn Z^T(y-Z\theta)+\widehat{D}\theta$$ and denote by $({\mathcal U}_r)_{r\ge 0}$ the family of sets
$$
{\mathcal U}_r = \left\{ \theta\in\Theta:\,
|S(\theta)|_\infty\leq\mu r
+\tau\right\}.
$$
We also define the function $\phi$ by
$$
\varphi(r)= \min_{\theta\in {\mathcal U}_r } |\theta|_1.
$$
We assume in the next theorem that the equation $r=\varphi(r)$ has a solution.
Note that $\varphi$ is decreasing on $[0,\infty)$ and $\varphi(r)\geq 0$. Moreover, for $r, r'\geq 0$ and $\alpha \in [0,1]$ we have $\alpha\mathcal{U}_r+(1-\alpha)\mathcal{U}_{r'}\subseteq \mathcal{U}_{\alpha r+(1-\alpha)r'}$ so that $\varphi$ is a convex function and therefore continuous in its domain. In particular, a solution exists provided $\varphi(0)<\infty$.\footnote{More generally, since $\varphi(r)<\infty \Leftrightarrow \mathcal{U}_r\neq\emptyset$, we can define $\underline{r}:=\inf\{ r \geq 0 : \varphi(r)<\infty \}$. A solution exists if and only if $\underline{r} \leq \varphi(\underline{r})$.}

We now present our algorithm.
Consider the following minimisation problem:
\begin{eqnarray}\label{CEquiv}
&&{\rm minimise} \ t \\
&& {\rm over} \ (t,\theta^+,\theta^-) \ {\rm such \ that:} \nonumber
\end{eqnarray}
$$
\theta^+-\theta^-\in\Theta,   \theta^+_j\geq 0, \theta^-_j\geq 0, \ j=1,\dots, p,
$$
$$
t = \sum_{j=1}^p(\theta^+_j+\theta^-_j),
$$
$$
\big|\mbox{$\frac{1}{n}$}Z^T(y-Z(\theta^+-\theta^-))+\widehat{D}(\theta^+-\theta^-)\big|_\infty\leq\mu t +\tau.
$$
Here the $\theta^+_j$ and $\theta^-_j$ are the components of $\theta^+$ and $\theta^-$ respectively. As previously, $\mu$ and $\tau$ are positive tuning constants, and $\Theta$ is a given subset of $\R^p$ characterising the prior knowledge about~$\theta$. Note that \eqref{CEquiv} is a convex program if $\Theta$ is a convex set, and it reduces to a linear program if $\Theta=\R^p$ or if $\Theta$ is defined by linear constraints.

Let $(\hat t,\hat\theta^+,\hat\theta^-)$ be a solution of \eqref{CEquiv}. We set $\widehat \t^{C'}=\hat\theta^+-\hat\theta^-$. The use of this algorithm is justified by the following theorem.

\begin{theorem}\label{lem1}
Assume that there exists a solution $\bar r$ to the equation $r=\varphi(r)$. Then $\widehat \t^{C'}$ is a solution of the minimisation problem (\ref{C}). Moreover, any solution $\widehat \t^{C}$ of (\ref{C}) induces a solution $(\bar r,\theta^+,\theta^-)$ of the problem (\ref{CEquiv}), where  $\theta^+$ and $\theta^-$ are vectors with components $\theta^+_j = \max\{\hat \theta^C_j,0\}$ and $\theta^-_j = \max\{-\hat\theta^C_j,0\}$.
\end{theorem}

The proof of Theorem \ref{lem1} is given in Section \ref{AppLem1} of this Supplementary Material. \\

We would like to emphasize that (\ref{CEquiv}) is not an obvious reformulation because the problem (\ref{C}) is non-convex.  The proof of Theorem \ref{lem1} exploits the structure of the $\ell_1$-norm regularisation. Again, recall that the rates attained by the compensated MU selector are suboptimal. However, it remains attractive compared to the conic programming estimator thanks to the simplicity of its computation.

\section{Proof of Theorem \ref{lem1}}\label{AppLem1}

Let $\bar r$ be a solution of the equation $r=\varphi(r)$. We set
$$
{\mathcal U}_* =  \left\{ \theta\in\Theta:\,
|S(\theta)|_\infty\leq\mu |\theta|_1
+\tau\right\}.
$$
The minimisation problem (\ref{C}) has the form
$$\min_{\theta\in {\mathcal U}_{*} } |\theta|_1.$$
First remark that
\begin{equation}\label{l11}
\min_{\theta\in {\mathcal U}_{*} } |\theta|_1\ge \bar r.
\end{equation}
Indeed, with the convention that the minimum over an empty set is equal to $+\infty$, we get
\begin{align*}
\min_{\theta\in {\mathcal U}_{*} } |\theta|_1&=\min\big(\min_{\theta\in {\mathcal U}_{*} :  |\theta|_1\le \bar r} |\theta|_1 ,  \min_{\theta\in {\mathcal U}_{*} :  |\theta|_1> \bar r} |\theta|_1\big)
\\
&\ge
\min\big(\min_{\theta\in {\mathcal U}_{\bar r} :  |\theta|_1\le \bar r} |\theta|_1 ,  \, \bar r \big)
\\
&\ge
\min\big(\min_{\theta\in {\mathcal U}_{\bar r} } |\theta|_1 ,  \, \bar r \big) = \min (\varphi(\bar r) ,  \, \bar r ) =  \bar r.
\end{align*}
Let now $\bar \t$ be any solution of
\begin{equation}\label{C1}
\min_{\theta\in {\mathcal U}_{\bar r} } |\theta|_1.
\end{equation}
Then $
 \bar\theta\in\Theta$, $|\bar \t|_1= \bar r$ and
$$|S(\bar \theta)|_\infty\leq\mu {\bar r}
+\tau = \mu |\bar \theta|_1
+\tau.$$
Thus $\bar \t\in {\mathcal U}_{*}$, which implies
\begin{equation*}\label{l13}
\min_{\theta\in {\mathcal U}_{*} } |\theta|_1\le |\bar \t|_1 = \bar r.
\end{equation*}
This and (\ref{l11}) imply that $\bar \t$ is also a solution of (\ref{C}) and
\begin{equation}\label{l12}
\min_{\theta\in {\mathcal U}_{*} } |\theta|_1 = \bar r.
\end{equation}
Hence all solutions of (\ref{C1}) are also solutions of (\ref{C}). Conversely, if $\theta'$ is a solution of  (\ref{C}), then, in view of (\ref{l12}), $|\theta'|_1=\bar r$. This and the fact that $\theta'\in {\mathcal U}_{*}$ imply that $\theta'\in\Theta$ and
$$|S(\theta')|_\infty\leq\mu {\bar r}
+\tau.$$ This means that $\theta'\in {\mathcal U}_{\bar r}$. Since $$\min_{\theta\in {\mathcal U}_{\bar r} } |\theta|_1=\bar r =|\theta'|_1,$$ we get that $\theta'$ is a solution of (\ref{C1}). Consequently, the solutions of (\ref{C}) and (\ref{C1}) coincide.

Let now $\hat \theta^C=(\theta^C_1,\dots,\theta^C_p)$ be a solution of (\ref{C}). Then setting $\theta^+_j = \max\{\hat \theta^C_j,0\}$, $\theta^-_j = \max\{-\hat\theta^C_j,0\}$, $t=|\theta^+|_1+|\theta^-|_1$, we have that $\hat\theta^C = \theta^+ - \theta^-$ and $|\hat \theta^C|_1 = t$. Thus, $(|\hat \theta^C|_1,\theta^+,\theta^-)$ is feasible for the problem (\ref{CEquiv}). This implies that the minimum in (\ref{CEquiv}) is smaller than the minimum in (\ref{C}), which yields $|\hat\theta^{C'}|_1 \leq t = |\hat\theta^C|_1$.
Moreover, for any solution $(\hat t, \hat \theta^+, \hat \theta^-)$ of (\ref{CEquiv}) the difference $\hat \theta^{C'}= \hat \theta^+ - \hat \theta^-$ satisfies
$$\begin{array}{rl}
\big|\mbox{$\frac{1}{n}$}Z^T(y-Z\hat\theta^{C'})+\widehat{D}\hat\theta^{C'}\big|_\infty& \leq\mu \hat t
+\tau \leq  \mu|\hat \theta^{C}|_1+\tau = \mu\bar r+\tau \end{array}$$
since $\varphi(\bar r ) = \bar r$.~Thus, $\hat\theta^{C'} \in \mathcal{U}_{\bar r}$.~Hence, by definition of $\varphi$, we have $\varphi(\bar r) \leq |\hat\theta^{C'}|_1$. Therefore, since we have shown before that $|\hat\theta^{C'}|_1\leq|\hat\theta^{C}|_1$, we obtain $|\hat\theta^{C'}|_1=|\hat\theta^{C}|_1=\bar r$, $\hat\theta^{C'}$ is a solution of (\ref{C}) and $(\bar r,\theta^+,\theta^-)$ is a solution of (\ref{CEquiv}).

\section{Properties of the sensitivities}\label{appendixC}

Here we collect some properties of the sensitivities $\kappa_{q}(s,u)$ and $\kappa_{\rm pr}(s,u)$. First, following~\cite{GT}, we  give a relation between $\kappa_{q}(s,u)$ and the Restricted Eigenvalue (RE) and Coherence (C) constants. For completeness, we recall the Restricted Eigenvalue
and Coherence assumptions.

\smallskip

{\bf Assumption RE($s,u$).} {\it Let $u>0$, $1\le s \le p$. There exists a
constant $\kappa_{\rm RE}(s,u)>0$ such that
$$
\min_{\D\in C_J(u)\setminus \{0\} }\frac{|\D^T\Psi \D|}{|\D_J|_2^2} \ge
\kappa_{\rm RE}(s,u),
$$
for all subsets $J$ of $\{1,\dots,p\}$ of cardinality $|J|\le s$.}

\smallskip

{\bf Assumption C.} {\it All diagonal elements of $\Psi$ are
equal to 1 and all its off-diagonal elements $\Psi_{ij}$ satisfy
the Coherence condition:
$\max_{i\neq j} |\Psi_{ij}|\le \rho$ for some $\rho<1$.}\\

Assumption C with $\rho<(cs)^{-1}$ and $c>0$ depending only on $u$ implies Assumption
RE($s,u$), see \cite{BRT09}. The following lemma due to \cite{GT} provides useful relations between
the constants $\kappa_{\rm RE}$, $\rho$ and $\kappa_q$. In this lemma, we denote by $c$ positive constants
that do not depend on $s$.
\begin{lemma}\label{lem:appC_Gautier_Tsyb} Let $u>0$, $1\le s\le p$. For any $\alpha\in (0,1)$, there exists $c>0$ such
that if Assumption C holds with $\rho<(cs)^{-1}$, then
\begin{equation}\label{k1}
\kappa_\infty(s,u)\ge \alpha.
\end{equation}
Next, under
Assumption RE($s,u$),
\begin{equation}\label{k3}
\kappa_{1}(s,u) \ge (cs)^{-1}\kappa_{\rm RE}(s,u)
\end{equation}
 and under
Assumption RE($2s,u$), for any
$s\le p/2$, $1< q\le 2$, we have
\begin{equation}\label{k4}
\kappa_{q}(s,u) \ge c(q) s^{-1/q}\kappa_{\rm RE}(2s,u),
\end{equation}
where $c(q)>0$ depends only on $u$ and $q$. Furthermore, for any $1\le
q\le\infty$,
\begin{equation}\label{k2}
\kappa_q(s,u) \ge (2s)^{-1/q}\kappa_\infty(s,u).
\end{equation}
\end{lemma}

 Note that \eqref{k1} and \eqref{k2} yield the control of the sensitivities $\kappa_q$ under the Coherence assumption for all
$1\le q\le\infty$. The next lemma relates $\kappa_{\rm pr}$ to $\kappa_{1}$.

\begin{lemma}\label{lem:appC1} For any $u>0$, $1\le s\le p$,
$$\kappa_{\rm pr}(s,u) \geq \sqrt{\kappa_1(s,u)}.$$
\end{lemma}
\begin{proof}
Fix a set $J$ such that $|J|\leq s$. Since $\Delta^T\Psi\Delta \le |\Psi\Delta|_\infty |\Delta|_1$, we obtain
\begin{align*}
 \min_{\Delta \in C_J(u):|\Psi^{1/2}\Delta|_2=1}|\Psi\Delta|_\infty  &= \min_{\Delta \in C_J(u):|\Psi^{1/2}\Delta|_2>0}|\Psi\Delta|_\infty/\sqrt{\Delta^T\Psi\Delta}\\
&\ge
 \min_{\Delta \in C_J(u):|\Psi^{1/2}\Delta|_2>0}\sqrt{|\Psi\Delta|_\infty/|\Delta|_1}\\
&\ge \min_{\Delta \in C_J(u):|\Delta|_1>0}\sqrt{|\Psi\Delta|_\infty/|\Delta|_1}\\
& =  \min_{\Delta \in C_J(u):|\Delta|_1=1}\sqrt{|\Psi\Delta|_\infty},
\end{align*}
\noindent where we used the fact that $\{\Delta : |\Psi^{1/2}\Delta|_2>0\} \subseteq \{ \Delta : |\Delta|_1 > 0, |\Psi\Delta|_\infty>0 \}$. Taking the minimum over $J$ such that $|J|\leq s$ and using the definitions of  $\kappa_{\rm pr}(s,u)$ and  $\kappa_1(s,u)$ we obtain the result.
\end{proof}

\begin{lemma}\label{lem:appC2} If ${\rm rank}(X)=\min\{n,p\}$, then for any $u>0$, $1\le s\le p$,
$$\kappa_{\rm pr}(s,u) > 0.$$
\end{lemma}
\begin{proof}
If ${\rm rank}(X)=p$, the result follows trivially, so we assume that ${\rm rank}(X)=n<p$. We have
\begin{align*}
& \min_{\Delta \in C_J(u):|\Psi^{1/2}\Delta|_2=1}|\Psi\Delta|_\infty  = \min_{\Delta \in C_J(u):|X\Delta/\sqrt{n}|_2=1}|X^TX\Delta/n|_\infty\\
&\geq \min_{\Delta \in \mathbb{R}^p :|X\Delta/\sqrt{n}|_2=1}|X^TX\Delta/n|_\infty
\geq \min_{\delta \in \mathbb{R}^n :|\delta|_2=1}|X^T\delta/\sqrt{n}|_\infty.
\end{align*}
Since ${\rm rank}(X^T)={\rm rank}(X)=n$, we have
$X^T\delta/\sqrt{n} \neq 0$ for all $\delta \in \mathbb{R}^n\setminus\{0\}$. Moreover, $\{\delta \in \mathbb{R}^n : |\delta|_2=1\}$ being compact, the minimum is achieved at some $\delta^*$, with $X^T\delta^*/\sqrt{n} \neq 0$, so that $|X^T\delta^*/\sqrt{n}|_\infty>0$. Taking the minimum over (the finite collection of) $J$ such that $|J|\leq s$ yields the result.
\end{proof}

\section{Kullback-Leibler divergence}
For $\theta\in\mathbb{R}^p$, recall that $c_{\theta}=\theta^T\Gamma\theta.$ Denote by $\lambda_{\min}$ and $\lambda_{\max}$ the smallest and largest eigenvalues of $\Gamma$. It is easy to see that they are functions of $\sigma^2_*$,  $\lambda_{\text{min}}^\Sigma$, and $\lambda_{\text{max}}^{\Sigma}$ only. Since the distribution of $X$ is now fixed, we write for brevity $\mathbb P_{X,\theta}=\mathbb P_{\theta}$. The following lemma is a crucial element in the proof of the lower bounds.

\begin{lemma}\label{lem:kullback}
Let $\theta_1\in \R^p$ and $\theta_2\in \R^p$ be such that $|\theta_1|_2=|\theta_1|_2$.
Under Assumption (A5),
$$
{\cal K}(\mathbb P_{\theta_1},\mathbb P_{\theta_2})\leq \frac{c\,n}{1+|\theta_1|_2^2}  \left(|\theta_1-\theta_2|_2^2+ |c_{\theta_1}-c_{\theta_2}|\right),
$$
where $c$ is a constant depending only on $\sigma^2_*$, $\sigma^2$, $\lambda_{\min}^\Sigma$, and $\lambda_{\max}^{\Sigma}$.
\end{lemma}
\begin{proof} In view of the properties of Kullback divergence between product measures, it suffices to prove the lemma for $n=1$.
In the following, we denote by $c$ positive constants depending only on $\sigma^2_*$, $\sigma^2$, $\lambda_{\text{min}}^\Sigma$, and $\lambda_{\text{max}}^{\Sigma}$, which may vary from line to line.
Let $\theta\in\mathbb{R}^p$.
Consider the random vector $(U,V)$ where
$$V=(A_1+B_1,\ldots,A_p+B_p),$$ with $A=(A_1,\ldots,A_p)^T$ a Gaussian vector with covariance matrix $\Sigma$
and $B=(B_1,\ldots,B_p)^T$ a Gaussian vector with covariance matrix $\sigma^2_*I_{p\times p}$, independent of $A$ and
$$
U=\sum_{j=1}^p\theta_{j}(V_j-B_j)+\eta,
$$
where $\eta$ is a zero-mean Gaussian random variable with variance $\sigma^2$. We now find the conditional distribution $\mathcal{L}_{\theta}(U|V)$ of $U$ given $V$.
Remark first that $(V_1,\ldots,V_p,B_1,\ldots B_p)^T$ is a zero-mean Gaussian random vector with covariance matrix
\[ \left( \begin{array}{cc}
\Sigma+\sigma^2_*I_{p\times p} & \sigma^2_*I_{p\times p} \\
\sigma^2_*I_{p\times p} & \sigma^2_*I_{p\times p}
\end{array} \right),\]
so that $\mathcal{L}_{\theta}(B|V)$ is a Gaussian distribution with mean $\tilde \Sigma V$ and covariance matrix
$\sigma^2_*(I_{p\times p}-\tilde \Sigma)$. This easily implies that $\mathcal{L}_{\theta}(U|V)$ is Gaussian with mean $\theta^T\Gamma V$ and variance $\sigma^2+c_{\theta}\sigma^2_*$. Then the logarithm of the density of $\mathcal{L}_{\theta}(U|V)$, denoted by $l_{\theta}(U|V)$ satisfies
$$
l_{\theta}(U|V)=-\frac{1}{2}\text{log}(2\pi)-\frac{1}{2}\text{log}(\sigma^2+c_{\theta}\sigma^2_*)-\frac{1}{2(\sigma^2+c_{\theta}\sigma^2_*)}(U-\theta^T\Gamma V)^2.$$
Let now $\theta_1\in\mathbb{R}^p$ and $\theta_2\in\mathbb{R}^p$ with
$|\theta_1|_2=|\theta_1|_2.$
Then,
\begin{align*}
l_{\theta_1}(U|V)-l_{\theta_2}(U|V)&=\frac{1}{2}\left(\text{log}(\sigma^2+c_{\theta_2}\sigma^2_*) -\text{log}(\sigma^2+c_{\theta_1}\sigma^2_*)\right)
\\&
+
\frac{1}{2(\sigma^2+c_{\theta_2}\sigma^2_*)}\big((U-\theta_2^T\Gamma V)^2-(U-\theta_1^T\Gamma V)^2\big)
\\&
+\left(\frac{1}{2(\sigma^2+c_{\theta_2}\sigma^2_*)}-\frac{1}{2(\sigma^2+c_{\theta_1}\sigma^2_*)}\right)(U-\theta_1^T\Gamma V)^2.
\end{align*}
Since the distribution of $V$ does not depend on $\theta$, we obtain that in the case $n=1$,
\begin{align*}
{\cal K}(\mathbb P_{\theta_1},\mathbb P_{\theta_2})&=\frac{1}{2}\left(\text{log}(\sigma^2+c_{\theta_2}\sigma^2_*) -\text{log}(\sigma^2+c_{\theta_1}\sigma^2_*)\right)
\\&
+
\frac{1}{2(\sigma^2+c_{\theta_2}\sigma^2_*)}\E_{\theta_1}\big((U-\theta_2^T\Gamma V)^2-(U-\theta_1^T\Gamma V)^2\big)
\\&
+\left(\frac{1}{2(\sigma^2+c_{\theta_2}\sigma^2_*)}-\frac{1}{2(\sigma^2+c_{\theta_1}\sigma^2_*)}\right)\E_{\theta_1}(U-\theta_1^T\Gamma V)^2,
\end{align*}
where $\E_{\theta_1}[\cdot]$ denotes the expectation when $\theta=\theta_1$ in the definition of $U$.
Using the inequality $|\text{log}(\sigma^2+x_1)-\text{log}(\sigma^2+x_2)| \le |x_1-x_2|/\min(\sigma^2+x_1, \sigma^2+x_2), x_1,x_2>0 $ and the fact that $c_{\theta}\ge \lambda_{\min}|\theta|_2^2$, we get
\begin{align*}
{\cal K}(\mathbb P_{\theta_1},\mathbb P_{\theta_2})&\le
 \frac{c}{1+|\theta_1|_2^2}\big|\E_{\theta_1}\big[(U-\theta_2^T\Gamma V)^2-(U-\theta_1^T\Gamma V)^2\big]\big|\\&+
  c\frac{|c_{\theta_1}-c_{\theta_2}|}{(1+|\theta_1|_2^2)^2}\left( (1+|\theta_1|_2^2) + \E_{\theta_1}[(U-\theta_1^T\Gamma V)^2] \right).
\end{align*}
We have
\begin{align*}
{\cal K}(\mathbb P_{\theta_1},\mathbb P_{\theta_2})&\leq \frac{c}{1+|\theta_1|_2^2}\Big|\E_{\theta_1}\Big[\big((\theta_1^T-\theta_2^T\Gamma)A-\theta_2^T\Gamma B\big)^2-\big((\theta_1^T-\theta_1^T\Gamma)A-\theta_1^T\Gamma B\big)^2\Big]\Big|
\\&
+c\frac{|c_{\theta_1}-c_{\theta_2}|}{(1+|\theta_1|_2^2)^2}
\big(1+ |\theta_1|_2^2+|\theta_1|_2^2\lambda_{\text{max}}^{\Sigma}+\sigma^2+|\theta_1|_2^2(\lambda_{\text{max}}^{\Sigma}+\sigma_{*}^2)\lambda_{\text{max}}^{2}\big).
\end{align*}
Hence
\begin{align*}
{\cal K}(\mathbb P_{\theta_1},\mathbb P_{\theta_2})&\leq \frac{c}{1+|\theta_1|_2^2}\Big|\E_{\theta_1}\Big[\big((\theta_1^T-\theta_2^T\Gamma)A-\theta_2^T\Gamma B\big)^2-\big((\theta_1^T-\theta_1^T\Gamma)A-\theta_1^T\Gamma B\big)^2\Big]\Big|
\\&+
c\frac{|c_{\theta_1}-c_{\theta_2}|}{1+|\theta_1|_2^2}.
\end{align*}
Then using the independence of $A$ and $B$, we get
\begin{align*}{\cal K}(\mathbb P_{\theta_1},\mathbb P_{\theta_2})&\leq \frac{c}{1+|\theta_1|_2^2}\big|(\theta_1^T-\theta_2^T\Gamma)\Sigma(\theta_1-\Gamma\theta_2)-(\theta_1^T-\theta_1^T\Gamma)\Sigma(\theta_1-\Gamma\theta_1)+(\theta_2^T\Gamma^2\theta_2-\theta_1^T\Gamma^2\theta_1)\big|\\&+c\frac{|c_{\theta_1}-c_{\theta_2}|}{1+|\theta_1|_2^2}.\end{align*}
Developing the preceding expression yields
\begin{align*}{\cal K}(\mathbb P_{\theta_1},\mathbb P_{\theta_2})&\leq \frac{c}{1+|\theta_1|_2^2}\big|2(\theta_1-\theta_2)^T\Gamma\Sigma\theta_1+\theta_2^T\Gamma(\Sigma+\sigma^2_*I_{p\times p})\Gamma\theta_2-\theta_1^T\Gamma(\Sigma+\sigma^2_*I_{p\times p})\Gamma\theta_1\big|\\&+c\frac{|c_{\theta_1}-c_{\theta_2}|}{1+|\theta_1|_2^2}.\end{align*}
The right hand side here can be rewritten as
$$\frac{c}{1+|\theta_1|_2^2}\Big(\big|(\theta_1-\theta_2)^T\Gamma\big(2\Sigma\theta_1-(\Sigma+\sigma^2_*I_{p\times p})\Gamma(\theta_1+\theta_2)\big)\big|+
|c_{\theta_1}-c_{\theta_2}|\Big).
$$
Since $(\Sigma+\sigma^2_*I_{p\times p})\Gamma=\Sigma$, we finally get
$${\cal K}(\mathbb P_{\theta_1},\mathbb P_{\theta_2})\leq \frac{c}{1+|\theta_1|_2^2}
\big(\big|(\theta_1-\theta_2)^T\Gamma\Sigma(\theta_1-\theta_2)\big|+
|c_{\theta_1}-c_{\theta_2}|\big),
$$
which implies the lemma.
\end{proof}

\section{Sufficient conditions for bounded $m_k$}\label{Sec5sm}

In this section, we provide conditions on the matrix $X$ guaranteeing that the quantity $m_k= \max_{j=1,\dots,p} \frac{1}{n}\sum_{i=1}^n |X_{ij}|^k$ concentrates around $\bar m_k= \max_{j=1,\dots, p} \frac{1}{n}\sum_{i=1}^n\mathbb{E}[|X_{ij}|^k]$. We are particularly interested in the cases $k=2$ and $k=4$ since such a concentration can be used to assure that assumptions of Corollary 1 hold.  However, the argument below is valid for any positive $k$. We assume that the independent random vectors $X_1,\ldots,X_n \in \R^p$ satisfy
\begin{itemize}
\item[(B)] $\mathbb{E}[ \max_{i\leq n}|X_i|_\infty^k]n^{-1}\log (p)=o(\bar m_k)$ as $n\to \infty$.
\end{itemize}
The condition is fairly mild and allows for $p>n$ under sub-exponential tail conditions. Sharper bounds are available in the Gaussian case (see remark below).

The argument is as follows. 
Lemma 9.1 in \cite{LPTvdG} with $m=1$ yields
$$
a:=\mathbb{E}[\max_{j\leq p}|\frac{1}{n}\sum_{i=1}^n (|X_{ij}|^k-\mathbb{E}[|X_{ij}|^k])|]
\leq  \sqrt{\frac{8\log (2p)}{n}} \mathbb{E}[ \max_{j\leq p} \{\frac{1}{n}\sum_{i=1}^n X_{ij}^{2k} \}^{1/2}].
$$
Thus, 
$$\begin{array}{rl}
a
& \leq \sqrt{\frac{8\log (2p)}{n}}  \mathbb{E}[ \max_{i\leq n}|X_i|_\infty^{k/2}\max_{j\leq p} \{\frac{1}{n}\sum_{i=1}^n |X_{ij}|^k \}^{1/2}]\\
& \leq \sqrt{\frac{8\log (2p)}{n}}  \mathbb{E}[ \max_{i\leq n}|X_i|_\infty^k]^{1/2}\{\mathbb{E}[\max_{j\leq p} \frac{1}{n}\sum_{i=1}^n |X_{ij}|^k]\}^{1/2}\\
& \leq  \delta_n (a +\bar m_k)^{1/2}
\end{array}$$
where $\delta_n =\sqrt{\frac{8\log (2p)}{n}}  \mathbb{E}[ \max_{i\leq n}|X_i|_\infty^k]^{1/2}$. This implies that $a\leq \delta^2_n+\delta_n \bar m_k^{1/2}$, and finally
$$
\begin{array}{c}\mathbb{E}[\max_{j\leq p}\frac{1}{n}\sum_{i=1}^n |X_{ij}|^k]  \leq \bar m_k +a \le \bar m_k + \delta^2_n+\delta_n \bar m_k^{1/2}.
\end{array}
$$
Therefore, by (B), we have $\mathbb{E}[m_k]=\bar m_k+o(\bar m_k)$ as the sample size $n$ grows.

~\\
{\bf Remark.} (Gaussian Case) {\it Let $X_1,\ldots,X_n \in \R^p$ be independent Gaussian random vectors such that $X_{ij}\sim {\cal N}(0,\sigma_j^2)$ for $j=1,\dots,p$. By \cite{LT1996}, page 21, equation (1.6), and the union bound, there exists a universal constant $\bar C\geq 1$ such that for any $k\geq 2$ and $\gamma \in (0,1)$
$$ \mathbb{P}\left[m_k^{1/k} \geq \bar C \sqrt{k} \max_{1\leq j\leq p}\sigma_j + n^{-1/k}\sqrt{2\log(2p/\gamma)}\max_{1\leq j\leq p}\sigma_j \right] \leq \gamma.$$
}

\end{document}